\documentclass{amsart}
\usepackage[utf8]{inputenc}
\usepackage{amsmath,amssymb,amsthm, graphics, comment,bm}
\usepackage{mathtools}
\usepackage{amsthm}
\usepackage{diagmac2}
\usepackage{textcomp}
\usepackage{color}
\usepackage[T1]{fontenc}
\usepackage[alphabetic]{amsrefs}
\usepackage[all,cmtip,color,matrix,arrow]{xy}
\usepackage{yfonts}
\usepackage{tikz}
\usepackage{tikz-cd}
\usepackage{stmaryrd}
\usepackage{mathrsfs}
\usepackage{enumerate}
\usepackage{hyperref}
\hypersetup{
	colorlinks=true,
	linkcolor=blue,
	filecolor=magenta,      
	urlcolor=cyan,
}
\usepackage[mathscr]{euscript}
\usepackage[all,cmtip]{xy}
\usepackage{cleveref}

\calclayout

\newcommand{\bA}{\mathbb{A}}

\newcommand{\bG}{\mathbb{G}}

\newcommand{\bQ}{\mathbb{Q}}

\newcommand{\bZ}{\mathbb{Z}}

\newcommand{\oH}{\operatorname{H}}

\newcommand{\sU}{\mathscr{U}}

\newcommand{\cB}{\mathcal{B}}

\newcommand{\cF}{\mathcal{F}}
\newcommand{\cG}{\mathcal{G}}
\newcommand{\cH}{\mathcal{H}}

\newcommand{\cL}{\mathcal{L}}

\newcommand{\cO}{\mathcal{O}}
\newcommand{\cP}{\mathcal{P}}

\newcommand{\cT}{\mathcal{T}}
\newcommand{\cU}{\mathcal{U}}

\newcommand{\cW}{\mathcal{W}}
\newcommand{\cX}{\mathcal{X}}
\newcommand{\cY}{\mathcal{Y}}
\newcommand{\cZ}{\mathcal{Z}}

\newcommand{\spec}{\operatorname{Spec}}

\newcommand{\Hom}{\operatorname{Hom}}

\newcommand{\Pic}{\operatorname{Pic}}

\newcommand{\Gm}{\bG_{{\rm m}}}
\newcommand{\GL}{\operatorname{GL}}
\newcommand{\SL}{\operatorname{SL}}

\newcommand{\Aut}{\operatorname{Aut}}

\newcommand{\Id}{\operatorname{Id}}

\newcommand{\Isom}{\operatorname{Isom}}
\newcommand{\uIsom}{\underline{\Isom}}

\newcommand{\bmu}{\bm{\mu}}
\newtheorem{theorem}{Theorem}[section]
\newtheorem{Teo}[theorem]{Theorem}
\newtheorem*{Teo*}{Theorem}
\newtheorem{Lemma}[theorem]{Lemma}

\newtheorem{Cor}[theorem]{Corollary}

\newtheorem{Prop}[theorem]{Proposition}
\newtheorem*{Ques*}{Question}

\theoremstyle{definition}

\newtheorem*{Oss'}{Remark}
\newtheorem{EG}[theorem]{Example}
\newtheorem{Def}[theorem]{Definition}
\newtheorem*{Def*}{Definition}
\newtheorem{Notation}[theorem]{Notation}
\newtheorem{Remark}[theorem]{Remark}

\begin{document}
\title[Effective morphisms and quotient stacks]{Effective morphisms and quotient stacks}
	\author[A. Di Lorenzo]{Andrea Di Lorenzo}
	\address[A. Di Lorenzo]{Humboldt Universit\"{a}t zu Berlin, Germany}
	\email{andrea.dilorenzo@hu-berlin.de}
	\author[G. Inchiostro]{Giovanni Inchiostro}
	\address[G. Inchiostro]{University of Washington, Seattle, Washington, USA}
	\email{ginchios@uw.edu}
	\maketitle
	\begin{abstract}
		We give a valuative criterion for when a smooth algebraic stack with a separated good moduli space is the quotient of a separated Deligne-Mumford stack by a torus. For doing so, we introduce a new class of morphisms, the \textit{effective} morphisms, which are a generalization of separated morphisms.
	\end{abstract}
\section{Introduction}
It has been clear since Grothendieck, that often there are advantages by considering objects as their functor of points. When one parametrizes objects with automorphisms,  it is natural to consider algebraic stacks.  As affine schemes are the building blocks for schemes, quotient stacks are the the building blocks for a wide class of stacks \cite{luna}.  Therefore, having some criterion for when an algebraic stack is a quotient stack is a foundational question which is worth investigating, both for purely theoretical reasons, and (as one expects) because it has some consequences on the moduli problem represented by the stack. For example, smooth quotient stacks have a well-defined notion of integral Chow ring \cites{EG}, and some times they measure the failure of $\operatorname{Br}= \operatorname{Br}'$ for schemes \cite{edidin2001brauer}.

While for schemes one has a cohomological criterion to determine if a scheme is affine, it is not at all obvious to come up with a criterion for when an algebraic stack is a quotient stack, nor that a criterion that completely characterizes quotient stacks should exist in the first place. 
Therefore in this paper we are interested in mildly relaxing the quotient assumption. In particular, we are interested in understanding when an algebraic stack $\cX$ is a quotient of a Deligne-Mumford stack $\cY$: one can understand Deligne-Mumford stacks as the the simplest stacks one encounters, with non-trivial stacky structure. 

In \Cref{section effective} we introduce \textit{effective} morphisms, which have the following properties:
\begin{enumerate}
\item generalize separated morphisms,
\item are defined by a valuative criterion, and
\item are stable under composition, base change, while being not \'etale local on the target.
\end{enumerate}
The relevance of effective morphisms lies in the fact that they can be used to recognize quotient of DM stacks by torus actions. What follows is our main result.
\begin{Teo}[\Cref{teo quotient implies effective} and \Cref{teo effettivo implica ho il gen set}]\label{teo intro}Let $\cX$ be a smooth algebraic stack and $\pi:\cX\to X$ a good moduli space morphism.  Assume that $X$ is separated, over a field $k$ of characteristic 0 containing all the roots of 1. Then $\pi$ is effective if and only $\cX$ is the quotient of a separated Deligne-Mumford stack by a split torus.
\end{Teo}
Effective morphisms are a generalization of separated morphisms, in the sense that all separated morphisms are effective, but effective morphisms are not \'etale local on the target. Observe that not being \'etale local on the target is essential for our application, as from \cite{luna}*{Theorem 6.4} any class of morphisms which is \'etale local on the target would not measure the property of being a quotient stack.

The complete definition of effective is \Cref{def:effective},  but on a first approximation one can understand it as follows. Consider $R'$ a DVR and $\spec(R')\cup_\eta \spec(R')$ the non-separated $\spec(R')$, namely the scheme obtained by gluing two copies of $\spec(R')$ along their generic points.
Assume that there is an action of $\bmu_2$ on $\spec(R')\cup_\eta \spec(R')$ which away from the origins is a Galois action,  and it swaps the two origins. We can take the quotient of this action, $P$. One can check that $P$ is non-separated, but it has a morphism $P\to \spec(R)$, where $R$ is the a DVR whose extension is $R'$, and such that the field extension $K(R)\subseteq K(R')$ has degree 2. An effective morphism of stacks $\cX\to \cY$ is a morphism satisfying that, for every separated Deligne-Mumford gerbe $\cT\to \spec(R)$, if we denote by $\cP:=P\times_{\spec(R)}\cT$ one has the following codimension one filling condition:
 $$\begin{tikzcd}\label{eq:diag eff}
        \cP\ar[r] \ar[d,] & \cX \ar[d] \\
        \cT\ar[r] \ar[ur,dotted] &\cY.
    \end{tikzcd}$$

The actual effectivity condition can be formulated as follows: one can consider \emph{any} \'{e}tale extension $R\subset R'$ and form the pushout
$P:=\spec(K(R)) \cup_{\spec(K(R'))} \spec(R')$; moreover, for any separated gerbe $\cT\to\spec(R)$, one can take the the cartesian product $\cP:= P\times_{\spec(R)} \cT$. A morphism of algebraic stacks $\cX\to \cY$ is effective if for every commutative diagram of solid arrows as (\ref{eq:diag eff}), one has a filling given by the dotted arrow above.


Other than the result in \Cref{teo intro}, effective morphisms are relevant in the context of the valuative criterion for universally closed morphisms, as we now briefly explain. Recall that if  $\cX\to \cY$ is a morphism of \textit{schemes} with mild assumptions, being proper can be measured with a codimension 1 filling property (namely, the valuative criterion for properness). It is well known that if $\cX$ and $\cY$ are algebraic stacks instead, being universally closed is equivalent to the following slightly modified valuative criterion. If one has a diagram as the one on the left, then there is a possibly ramified extension of DVRs $\spec(R)\to \spec(R')$ and a diagram as the one on the right:
$$\xymatrix{\spec(K(R))\ar[r] \ar[d] & \cX\ar[d] \\ \spec(R) \ar[r] & \cY} \text{ }\text{ }\text{ }\text{ }\xymatrix{\spec(K(R'))\ar[d] \ar[r] & \spec(K(R))\ar[r] \ar[d] & \cX\ar[d] \\ \spec(R')\ar@{..>}[rru]\ar[r]&\spec(R) \ar[r] & \cY.} $$
It is very natural to ask what conditions one can impose on the extension $\spec(R')\to \spec(R)$, and still guarantee the extension of a lifting as before. We prove the following:
\begin{Teo}[\Cref{teo val criterion}]
    Assume that $\cX\to \cY$ is effective and universally closed morphism, over a field $k$ of characteristic 0 containing all the roots of 1. Then, in the diagram above, one can assume that $R'=R[t]/(t^n-\pi)$ where $\pi$ is a uniformizer for $R$. In particular, one can assume that the residue fields of $R$ and $R'$ are isomorphic.
\end{Teo}
See also \cite{bresciani2022arithmetic} for an analogous result in the case where $\cX$ is a tame Deligne-Mumford stack.

The paper is organized as follows. In \Cref{section effective} we introduce effective morphisms, and we prove that if $\cX$ is an algebraic stack with a good moduli space admitting a separated, relatively Deligne-Mumford morphism $\cX\to \cB \Gm^n$, then the good moduli space map is effective (\Cref{teo quotient implies effective}). In \Cref{section torsors and twisted tori} we report some results about torsors and twisted tori; in particular we prove that the classifying space of a twisted torus is a certain fiber product (\Cref{theorem from semidirect product to nonsplit torus}), which might be of independent interest.
In \Cref{section when BG not eff} we prove that $\cB G$ is effective if and only if $G$ is a central extension of a finite group by a split torus, whereas in \Cref{section from eff to quotient} we prove the other direction of \Cref{teo intro}

\subsection{Conventions}
All algebraic stacks will have affine (hence separated) diagonal. We work over a field $k$ which is of characteristic 0 containing all the roots of 1. Given two morphisms $X\to \spec(A)$ and $\spec(B)\to \spec(A)$, we denote by $X_B:=X\times_{\spec(A)}\spec(B)$. Unless oterwise stated, all groups will be reductive and over $k$, and all actions will be left actions. All algebraic stacks will be essentially of finite type over $k$, and the automorphism group of their geometric points will be reductive. We will use the shorthand DM for Deligne-Mumford. 
\subsection*{Acknowledgements} We are thankful to Jarod Alper, Dori Bejleri, Siddarth Mathur, Minseon Shin, David Rydh and Angelo Vistoli for helpful conversations. We are thankful to the referee for carefully reading this manuscript and for pointing out a gap in a previous version of this work. 

\section{Effective morphisms}\label{section effective}

In this section we introduce the main class of morphisms of algebraic stacks, \textit{effective morphisms}, which are a mild generalization of separated morphisms.

\subsection{The ``local bug-eyed cover”}We recall a construction of a special push-out, related
to the ``bug-eyed cover” (\cite[Example 2.19]{kollar1997quotient} and \cite[Example 3.9.2]{Alp21}). 

Consider a DVR $R$ with fraction field $K(R)$, let $K(R')$ be a finite extension of $K(R)$ and let $R'$ be a DVR with fraction field $K(R')$ such that the inclusions $R\to K(R) \to K(R')$ lands in $R'$; namely, the ring $R'$ is the localization along a closed point of the normalization of $\spec(R)$ in $K(R')$. Assume also that $\spec(R')\to \spec(R) $ is \'etale. We have the following diagram:
$$\xymatrix{\spec(K(R'))\ar[r] \ar[d] & \spec(K(R))\ar[d] \\ \spec(R')\ar[r] & \spec(R). }$$
Recall that as $R\to R'$ is \'etale, the uniformizer of $R$ is also a uniformizer for $R'$. If we denote by $k_R$ and $k_{R'} $ the residue fields of $R $ and $R'$ respectively, we have an extension $k_R\subseteq k_{R'}$. 

It follows from \cite[Theorem B]{rydh2011etale} that if this extension is surjective, the diagram above is a pushout in the category of algebraic stacks. In other terms, for every algebraic stack $\cX$, the data of a map $\spec(R)\to \cX$ is equivalent to two maps $a:\spec(R')\to \cX$ and $b:\spec(K(R))\to \cX$, together with an isomorphism $\sigma:a|_{K(R')}\to b|_{K(R')}$. However, if $k_R\subsetneq k_{R'}$, the two maps $\spec(K(R'))\to \spec(K(R))$ and $\spec(K(R'))\to\spec(R')$ admit a pushout $P:=P_{R,R'}$ (from \cite{rydh2011etale}) which is \textit{not} isomorphic to $\spec(R)$, as the residue field of the closed point of $P$ is $k_{R'}$.
\begin{Def}
    We call the \textit{local bug-eyed cover} the pushout $P_{R,R'}$ described above. When there is no ambiguity, we drop the subscripts $R$ and $R'$.
\end{Def}
\begin{EG}
We describe the push-out $P$ in the special case where $[K(R):K(R')] = 2$ (and $k_R \neq k_{R'}$). In this case there is a $\bmu_2$-action on $\spec(R')$, whose associated quotient scheme is $\spec(R)$. We can pull back the cartesian diagram on the left along $\spec(R')\to\spec(R)$ to obtain the one on the right:
\[
\begin{tikzcd}
\spec(K(R')) \ar[r] \ar[d] & \spec(R') \ar[d] \\
\spec(K(R)) \ar[r] & P,
\end{tikzcd}
\quad
\begin{tikzcd}
\spec(K(R'))\times\bmu_2 \ar[r] \ar[d] & \spec(R')\times\bmu_2 \ar[d] \\
\spec(K(R')) \ar[r] & P'.
\end{tikzcd}
\]
Observe that $P'$ is still a pushout, because formation of the pushout commutes with base change \cite[Theorem C]{rydh2011etale}.
By construction, the closed points of $P'$ correspond to two copies of $\spec(k_{R'})$: we can then regard $P'$ as two copies of $\spec(R')$ glued at the generic points. It follows then that $P$ is the algebraic space obtained by quotiening $P'$ by the action of $\bmu_2$ that on the generic point agrees with the Galois action, and it switches the two closed points.
\end{EG}
\subsection{Effective morphisms}
We now introduce the main definition of this
paper.
\begin{Def}\label{def:effective} Let $\cX\to\cY$ be a morphism of algebraic stacks, let $R$ (resp. $R')$ be a DVR with fraction field $K(R)$ (resp. $K(R')$) and $\beta:\spec(R')\to\spec(R)$ be an \'etale extension. Let $P:=P_{R,R'}$ be the corresponding local bug-eyed cover. Then we say that $\cX\to\cY$ is \emph{effective} if:
given any $R$ and $R'$ as above, given any separated DM gerbe $\cT\to \spec(R)$, and a diagram of solid arrows as the one below, one can always find a dotted arrow such that the resulting diagram commutes:
    \[
    \begin{tikzcd}
        P\times_R \cT \ar[r] \ar[d] & \cX \ar[d] \\
        \cT\ar[r] \ar[ur, dotted] & \cY.
    \end{tikzcd}
    \]
    As usual, if $k$ denotes the ground field, we say that $\cX$ is effective if $\cX\to \spec(k)$ is effective.
\end{Def}
\begin{Remark}It is already interesting to specialize the definition of effective to the case where $\cT=\spec(R)$, namely, to the case where the gerbe is trivial. In this case, one can check that for any morphism $f:\spec(R)\to \cY$ such that
\begin{enumerate}
\item we can lift $f$ generically, i.e. there exists a lifting $f_{K(R)}:\spec(K(R))\to\cX$,
\item we can lift $f$ globally up to \'etale covers, in the sense that there exists an \'etale extension $\beta:\spec(R')\to\spec(R)$ and a lifting $g:\spec(R')\to\cX$, and
\item there is an isomorphism $\varphi:(\beta|_{K(R')})^*f_{K(R)}\to g|_{K(R')}$
\end{enumerate}
then we can actually lift $f$ to a morphism $\spec(R)\to\cX$, in a way compatible with $f_{K(R)}$, $g_{K(R')}$ and $\varphi$. This situation arises typically when $\cX\to\cY$ is universally closed, thanks to the existence part of the valuative criterion of properness for algebraic stacks \cite{stacks-project}*{\href{https://stacks.math.columbia.edu/tag/0CLW}{Tag 0CLW},\href{https://stacks.math.columbia.edu/tag/0CLX} {Tag 0CLX}}.
\end{Remark}
\begin{EG} The classifying stack $\cB \GL_2$ is not effective. We will see in \Cref{section when BG not eff}, in particular \Cref{Teo BG effective implies G is a central ext of a torus} that $\cB G$ is effective if and only if $G$ is a central extension of a split torus by a finite group. This can also be proven directly, using the same argument of \Cref{prop sl2 not effective}.
\end{EG}
It turns out that effective morphisms are a generalization of separated morphisms.
\begin{Prop}\label{prop:sep to eff}
    Assume that $\cX\to\cY$ is separated. Then it is also effective.
\end{Prop}
We will first prove a few auxiliary lemmas, and introduce the following notations
\begin{Notation}
    In this section, $R\to R'$ will denote an \'etale extension of DVRs, $\cT\to \spec(R)$ will be a gerbe which is a DM stack, $P=P_{R,R'}$ will be the local bug-eyed cover associated to $R\to R'$. Furthermore:
    \begin{itemize}
        \item we denote by $\cT' :=\spec(R')\times_{\spec(R)}\cT$,
        \item we denote by $\cU:=\spec(K(R))\times_{\spec(R)}\cT$ and $\cU':=\spec(K(R'))\times_{\spec(R)}\cT$,
        \item we denote by $\cP:= P\times_{\spec(R)}\cT$.
    \end{itemize}
    \end{Notation}

\begin{Lemma}\label{lemma_extension_isom_from_open_of_regular_onedim}
Let $\cX\to \cY$ be a separated morphism of algebraic stacks, let $T$ be a regular 1-dimensional scheme with a morphism $\phi: T\to \cY$, and let $U\hookrightarrow T$ be a dense open subscheme of $T$. Assume that one has two lifts $f_1, f_2:T\to \cX$ of $\phi$ and an isomorphism $\xi:f_1|_U\to f_2|_U$. Then there is an unique extension of $\xi$ to an isomorphism $f_1\to f_2$.
\end{Lemma}
\begin{proof}
The setup is equivalent to the following diagram
\begin{equation}
    \begin{tikzcd}
        \uIsom(f_1|_U,f_2|_U) \ar[r] \ar[d, "q"] & \uIsom(f_1,f_2) \ar[r] \ar[d, "p"] & \cX \ar[d, "\delta_{\cX/\cY}"] \\
        U \ar[r] & T \ar[r] & \cX\times_{\cY}\cX
    \end{tikzcd}
\end{equation}
and, by assumption, $q$ has a section. But $\delta_{\cX/\cY}$ is affine and proper, so also $p$ is affine and proper. By assumption, $T$ is regular of dimension one, so from the valuative criterion for properness every morphism defined on a dense open subscheme of $T$ extends uniquely to the whole $T$. In particular, as $q$ has a section, $p$ is proper, and $T$ is regular and of dimension 1, also $p$ has a section. In other terms, $f_1$ and $f_2$ are isomorphic, and as $p$ is also separated, the isomorphism is unique.
\end{proof}

\begin{Lemma}\label{lemma_separated_implies_effective_with_no_gerbe}
Let $T$ be a regular 1-dimensional scheme, let $U\to T$ be a schematically dense open subscheme of $T$ and let $T'\to T$ be an \'etale and surjective morphism. Assume that one has a commutative diagram of solid arrows as the one below, where $\cX\to \cY$ is a separated morphism and $Q$ is the push-out of $U'\to U$ and $U'\to T'$ which exists by \cite{rydh2011etale}*{Theorem B}:
$$\begin{tikzcd}
      Q \ar[r] & \cX \ar[d] \\
     T \ar[r] \ar[from=u, crossing over] \ar[ur, dotted] & \cY.
\end{tikzcd}$$.

Then there is a dotted arrow as the one above, which makes the diagram 2-commutative. 
\end{Lemma}
\begin{proof}
From the definition of pushout, the diagram in the statement of \Cref{lemma_separated_implies_effective_with_no_gerbe} is equivalent to the following diagram
\begin{equation}
\begin{tikzcd}
    U'=U\times_T T'\ar[r, "j"]\ar[d] & T' \ar[r, "\xi"] & \cX \ar[d] \\
    U \ar[r] \ar[urr] & T \ar[r] \ar[from=u, crossing over] \ar[ur, dotted] & \cY.
\end{tikzcd}
\end{equation}

Our strategy consists in making the arrow $\xi:T'\to\cX$ descend along the \'etale cover $T'\to T$. For this, consider the projections
\[\begin{tikzcd} \pi_1,\pi_2:T' \times_{T} T' \ar[r] & T'
\end{tikzcd}\]
\[
\begin{tikzcd}
    \rho_1,\rho_2:U'\times_{U}U'\ar[r] & U'
\end{tikzcd}
\]
respectively on the first and the second factor.  By assumption,  the morphism $j\circ\xi$ descends along $U'\to U$, hence there is an isomorphism $\sigma_U:\rho_1^*j^*\xi\to \rho_2^*j^*\xi$.
Observe that $U'\times_U U'\to T'\times_T T'$ is schematically dense. Indeed, the property of an open subscheme to be dense is preserved by \'etale base change, so as $U'\to T'$ is dense, also $U'\times_T T' = U'\times_UT\to T'\times_T T'$ is schematically dense. Similarly, the composition of dense
open embeddings is a dense open embedding, so also $U'\times_U U'\to U'\times_U T'\to T'\times_T T'$ is dense. Therefore from \Cref{lemma_extension_isom_from_open_of_regular_onedim} the isomorphism $\sigma_U$ extends to $\sigma:\pi_1^*\xi\to \pi_2^*\xi$.  As $\sigma_U$ satisfies the cocycle condition, from the uniqueness in  \Cref{lemma_extension_isom_from_open_of_regular_onedim}  also $\sigma$ satisfies the cocycle condition. Hence the morphism $\xi$ descends as desired.\end{proof}

\begin{proof}[Proof of \Cref{prop:sep to eff}]
We begin by introducing some notations:
\begin{itemize}
\item let $T\to \cT$ be an \'etale atlas,
\item let $T':=\spec(R')\times_{\spec(R)}T$ and $T_P:= P\times_{\spec(R)}T$, 
\item let $U:=T\times_{\spec(R)} \spec(K(R))$ and $U':=R\times_{\spec(R)} \spec(K'(R))$.
\end{itemize}

Recall that (see for example \cite[proof of Proposition 4.18]{LMB}), if $\cW_1$ and $\cW_2$ are algebraic stacks and $W\to \cW_1$ is a \'etale cover, a map $f:\cW_1\to \cW_2$ is equivalent to a map $f_W:W\to \cW_2$ with an isomorphism between the two maps $\pi_1\circ f_W, \pi_2\circ f_W:W\times_{\cW_1} W\to \cW_2$ which satisfies the cocycle condition.
Hence the morphism $\cP\to \cX$ is equivalent to a morphism $f_P:T_P\to \cX$ with an isomorphism $\xi_P:\pi_1\circ f_P \to \pi_2\circ f_P$ that satisfies the cocycle condition. 

Observe now that, as the formation of the push-out $P$ commutes with base change \cite{rydh2011etale}*{Theorem C} we have that $T_P$ is the push-out of the open embedding $U'\to T'$ and the \'etale morphism $U'\to U$.  Observe also that:
\begin{enumerate}
\item as $\spec(K(R))\to \spec(R)$ is dense and $T\to \spec(R)$ is \'etale, also $U\to T$ is dense, and
\item as $\spec(R')\to \spec(R)$ is \'etale, also $U'\to U$ is \'etale.
\end{enumerate}
Then the assumptions of \Cref{lemma_separated_implies_effective_with_no_gerbe} apply,  so the map $T\to \cY$ lifts to $\phi:T\to \cX$.  

The two maps $\phi\circ\pi_1,\phi\circ \pi_2:T\times_\cT T \to T\to \cX$ are isomorphic when restricted to $U\times_\cU U$ as they descend to $\cU$, and $U\times_\cU U \to Z\times_\cZ Z$ is open and dense. Hence by \Cref{lemma_extension_isom_from_open_of_regular_onedim}, there is an isomorphism $\sigma:\phi\circ\pi_1\to \phi\circ \pi_2$. It satisfies the cocycle condition as it satisfies it once restricted to $U\times_\cU U$, and from the uniqueness in \Cref{lemma_extension_isom_from_open_of_regular_onedim}.  Hence there is a morphism $\cT\to \cX$.
\end{proof}

\begin{Cor}\label{cor hom from spec r sono determinati da hom from P}
    For every algebraic stack $\cX$, the map $Hom(\cT,\cX)\to Hom(\cP,\cX)$ is fully faithful.
\end{Cor}
\begin{proof}Consider two morphisms $f,g:\cT\to \cX$ which are isomorphic when we restrict them to $\cP$. They give rise to the following diagram
$$\begin{tikzcd}
        \uIsom(f|_{\cP},g|_{\cP}) \ar[r] \ar[d,"q"] & \uIsom(f,g) \ar[r] \ar[d, "p"] & \cX \ar[d, "\delta_{\cX/\cY}"] \\
        {\cP} \ar[r] &\cT \ar[r,"f\times g"] & \cX\times_{\cY}\cX.
    \end{tikzcd}$$
    Since $f|_{\cP}\cong g|_{\cP}$, the map $q$ has a section. Then $p$ has a section when we pull it back to $\cP$. Since $p$ is affine, it is separated, so it is effective. Then the map $p$ has a section. 
\end{proof}
\begin{Cor}\label{cor composition effective implies the first one effective}
Assume that the composition $X\to Y \to Z$ is effective. Then $X\to Y$ is effective.
\end{Cor}
\begin{proof}Consider the following diagram on the left:
$$\xymatrix{\cP\ar[d]_\pi \ar[r] ^a & X\ar[d]^p&\\ \cT\ar[r]^-b &Y\ar[r] & Z} \text{ }\text{ }\text{ }\text{ }\text{ }\text{ }\text{ }\text{ }\xymatrix{\cP\ar[d]_\pi \ar[r] ^a & X\ar[d]^p\\ \cT\ar[r]_-b \ar[ru]^-f&Y}$$
Since $X\to Z$ is effective, there is a map $f:\cT\to X$ such that $a$ and $ f\circ \pi$ are isomorphic. But then $p\circ a$ and $p\circ f \circ \pi$ are isomorphic, and $p\circ a$ is isomorphic to $b\circ \pi$. So by \Cref{cor hom from spec r sono determinati da hom from P} the maps $p\circ f $ and $b$ are isomorphic. In particular the diagram on the right is commutative, so $p$ is effective.\end{proof}
Non separated morphisms can be non-effective, even in the case of algebraic spaces.
\begin{EG}
    Let $R=\mathbb{C}(x^2)[y]_{(y)}$ and consider the \'etale extension of DVRs
    \[R\longrightarrow \mathbb{C}(x)[y]_{(y)}=:R'.\]
    Let $P$ be the pushout of $\spec(K(R'))\to\spec(K(R))$ and $\spec(K(R'))\to\spec(R')$. Then $P\to\spec(R)$ is a morphism of algebraic spaces that is not effective: indeed, if it were effective, we would have a section of the field extension $\mathbb{C}(x^2) \hookrightarrow \mathbb{C}(x)$, which is not the case.
    
    More generally, in order for the morphism $P\to \spec(R)$ to be effective, the residue field of the closed point of $P$ should be $k_R$. This is clearly necessary as otherwise the map
    $P\to \spec(R)$ would not have a section - as before. It is sufficient as if the residue field $k_{R'}$ is isomorphic to $k_R$, then $P=\spec(R)$ from \cite{rydh2011etale}*{Theorem B}. So in particular, $P\to \spec(R)$ is effective if and only if the residue field $k_{R'}\cong k_R$ if and only if $P\to \spec(R)$ is an isomorphism.
\end{EG}
\begin{Lemma}\label{lm:composition}
    Composition of effective morphisms is effective, and being effective is stable under base change.
\end{Lemma}
\begin{proof}
    The first statement is straightforward, we prove the second one.
    For any effective morphism $f:\cX\to\cY$ and any morphism $g:\cW\to\cY$, we have a 2-commutative diagram of solid arrows
    \[
    \begin{tikzcd}
        \cP \ar[r] \ar[d, "p"] & \cX \times_{\cY} \cW \ar[r] & \cX \ar[d, "f"] \\
        \cT \ar[r] \ar[urr, dotted] & \cW \ar[from=u, crossing over] \ar[r, "g"] & \cY
    \end{tikzcd}
    \]
    Effectiveness of $f$ implies that there is a lifting given by the dotted arrow above. In this way, we have well defined morphisms $\xi:\cT\to\cX$ and $\zeta:\cT\to\cW$; to obtain a morphism $\cT\to\cX\times_{\cY}\cW$, the only missing data is an isomorphism $\alpha:f(\xi)\simeq g(\zeta)$ or, in other terms, a section of $\uIsom_{\cT}(f(\xi),g(\zeta))$. 

    The morphism $\cP\to \cX\times_{\cY}\cW$ corresponds to the data $p^*\xi$, $p^*\zeta$ and an isomorphism $p^*\xi \simeq p^*\zeta$. In particular, we have a commutative diagram of solid arrows.
    \[
    \begin{tikzcd}
        \cP \ar[r] \ar[d] & \uIsom_{\cT}(f(\xi),g(\zeta)) \ar[d, "q"] \\
        \cT \ar[r, "\Id"] \ar[ur, dotted] & \cT
    \end{tikzcd}
    \]
    The morphism $q$ is affine because by hypothesis $\cY$ has affine diagonal, hence it is separated and thus effective by \Cref{prop:sep to eff}. This means that we have a lifting given by the dotted arrow in the diagram above, which is the desired section of $\uIsom_{\cT}(f(\xi),g(\zeta))$.  
\end{proof}
We now make a short digression on how effective morphisms are related to the valuative criterion for universally closed morphisms. Recall that if one has an universally of stacks $\cX\to \cY$, then it satisfies the \textit{stacky }valuative criterion for universally closed morphisms \cite{stacks-project}*{\href{https://stacks.math.columbia.edu/tag/0CLW}{Tag 0CLW},\href{https://stacks.math.columbia.edu/tag/0CLX} {Tag 0CLX}}, namely if one has a diagram as the one on the left, then there is a possibly ramified extension of DVRs $\spec(R)\to \spec(R')$ and a diagram as the one on the right:
$$\xymatrix{\spec(K(R))\ar[r] \ar[d] & \cX\ar[d] \\ \spec(R) \ar[r] & \cY} \text{ }\text{ }\text{ }\text{ }\xymatrix{\spec(K(R'))\ar[d] \ar[r] & \spec(K(R))\ar[r] \ar[d] & \cX\ar[d] \\ \spec(R')\ar@{..>}[rru]\ar[r]&\spec(R) \ar[r] & \cY.} $$
It is natural to wonder if, in some specific cases, one can impose some restrictions on the extension $R\subseteq R'$.
\begin{Teo}\label{teo val criterion}
    Assume that $\cX\to \cY$ is effective and universally closed. Then one can assume that $R'=R[t]/(t^n-\pi)$ where $\pi$ is a uniformizer for $R$.
\end{Teo}
\begin{proof}
 Let $\spec(R')\to \spec(R)$ an extension such that the map $\spec(K(R'))\to\spec(K(R))\to \cX$ extends along $\spec(R')$, and let $n$ be the ramification index. Consider $\widetilde{R}=R[t]/(t^n - \pi)$. Then $R\hookrightarrow \widetilde{R}$ is totally ramified of order $n$ and by Abhyankar's lemma \cite[\href{https://stacks.math.columbia.edu/tag/0BRM}{Tag 0BRM}]{stacks-project} the localizations of the normalization of $R'\otimes_{R}\widetilde{R}$ are \'etale over $\widetilde{R}$. But since $\cX\to \cY$ is effective, we can extend the map $\spec(K(\widetilde{R}))\to \spec(K(R))\to \cX$ to a map $\spec(\widetilde{R})\to \cX$ as desired.
\end{proof}
\subsection{Preparatory lemmas on Picard groups}
We collect here some technical lemmas that will be needed later.
Recall that a finite group $G$ over a field $k$ is \emph{split} when $|G(k)|=|G(\overline{k})|$, whether a split torus over $k$ is an algebraic group isomorphic to $\bG_{\rm{m},k}^n$. From now one we will use the simplified notation $\Gm^n$ for split tori.
\begin{Lemma}\label{lemma_pic_remains_the_same_after_field_ext}
    Let $k\to k'$ a field extension, and let $G$ be a group over $k$. Assume that $G$ is a central extension of a split torus $\Gm^n$ by a finite split group. Then the pull-back map $\Pic(\cB G)\to \Pic(\cB G_{k'})$ is an isomorphism.
\end{Lemma}
\begin{proof}
    The proof relies on the fact that, by assumption, $k$ is of characteristic 0 and contains the roots of 1. This guarantees that the result is true if $G$ is finite split, or if $G\cong \Gm^n$ for some $n$, or if $G$ is of the form $\Gm^n\times F$ for some finite group $F$. Indeed, a line bundle on $\cB G$ is a homomorphism $G\to \Gm$ and, for the aforementioned groups, the set of all such homomorphisms remains invariant under base change.

    In general, from \cite{brionext}, there is a finite subgroup $F$ of $G$, and a surjective homomorphism $\Gm^n\rtimes F \to G$ with finite kernel $K$. Observe that, as $\Gm^n$ is central, the semidirect product is a product. In other terms, $G$ fits in a sequence as follows:
    $$1\to K\to \Gm^n\times F\to G\to 1.$$
    We have then the following diagram:
    $$\xymatrix{0\ar[r] & \Hom(G,\Gm)\ar[r]\ar[d] &\Hom(\Gm^n\times F,\Gm)\ar[d]^\cong\ar[r] & \Hom(K,\Gm)\ar[r] \ar[d]^\cong & \\0\ar[r] & \Hom(G|_{k'},\Gm|_{k'})\ar[r] &\Hom(\Gm^n|_{k'}\times F|_{k'},\Gm|_{k'})\ar[r] & \Hom(K|_{k'},\Gm|_{k'})\ar[r] &}  $$
    The desired result follows from a diagram chase.
\end{proof}

Let us recall our notation. We set $R$ to be a DVR with fraction field $K(R)$ and residue field $k_R$, and we set $\cT\to \spec(R)$ to be a gerbe which is a DM stack. Furthermore, we set $\cU:=\spec(K(R))\times_{\spec(R)}\cT$.
\begin{Lemma}\label{pic_gerbe_depends_only_on_Pic(Z)}
The restriction maps $\Pic(\cT)\to \Pic(\cU)$ is an isomorphism.
\end{Lemma}
\begin{proof}
We need to show that the pullback map is both injective and surjective.

\underline{Surjective.} This follows as we can always extend line bundles from an open non-empty substack of a smooth algebraic stack to the whole algebraic stack.

\underline{Injective.} By definition of gerbe, there exists an \'{e}tale cover of $\spec(R)$ such that the pullback of $\cT$ is a classifying stack. We can assume that such \'{e}tale cover is connected and that it has only one closed point, hence given by $\spec(A)\to\spec(R)$ with $A$ local. 

Moreover, we have $\cT\times_{\spec(R)} \spec(A) \simeq \cB \cG$ for a finite \'{e}tale group scheme $\cG\to\spec(A)$. Up to taking a refined cover, we can assume that $\cG\simeq \spec(A)\times G$ for a constant (hence split) finite group scheme $G$, hence $\cB \cG \simeq \spec(A)\times\cB G$. Observe that $\cU\times_{\spec(R)} \spec(A)$ is isomorphic then to $\spec(K(A))\times\cB G$. 

We have a commutative diagram of pullback homomorphisms
\[
\begin{tikzcd}
    \Pic(\cT) \ar[r] \ar[d] & \Pic(\cU) \ar[d] \\
    \Pic(\spec(A)\times\cB G) \ar[r] & \Pic(\spec(K(A))\times\cB G)
\end{tikzcd}
\]
hence if we prove that the left vertical arrow and the bottom horizontal arrow are injective, we are done.

Observe that the composition $\Pic(\cT) \to \Pic(\spec(A) \times\cB G) \to \Pic(\cB G_{\overline{k}_R})$ (where the last map is the restriction to any geometric closed point of $\spec(A) \times\cB G$)  is injective: indeed by \cite{alp}*{Theorem 10.3} every line bundle on $\cT$ whose restriction to the residual gerbes of every geometric closed point is trivial must come from the good moduli space, which is $\spec(R)$ in our case; but every line bundle on $\spec(R)$ is trivial, from which our conclusion follows. This also implies that $\Pic(\cT) \to \Pic(\spec(A) \times\cB G)$ is injective.

For the second arrow, observe that by \Cref{lemma_pic_remains_the_same_after_field_ext} the composition $\Pic(\cB G)\to\Pic(\spec(A)\times\cB G)\to \Pic(\cB G_{K(A)})$ is an isomorphism, hence $\Pic(\spec(A) \times \cB G) \to \Pic(\cB G_{K(A)})$ is injective, thus concluding the proof.
\end{proof}
\subsection{Effectiveness of quotients of DM stacks by a torus}
We now prove one direction of \Cref{teo intro}.
\begin{Prop}\label{prop_Gm_effective_group}
    The morphism $\cB\Gm^n\to\spec(k)$ is effective.
\end{Prop}

\begin{proof}
    First we prove the case $n=1$.
Given a diagram of solid arrows as the one below, we want to construct a dotted arrow $g$:
    \begin{equation}\label{eq:BGm eff diag}
    \begin{tikzcd}
        \cP\ar[r] \ar[d, "p"] & \cB\Gm \ar[d] \\
        \cT \ar[r] \ar[ur,dotted, "g"] & \spec(k)
    \end{tikzcd}
    \end{equation}
    
    The data of a map $\cP\to\cB\Gm$ corresponds to a triple $(M_1,M_2,\varphi)$ where $M_1$ (resp. $M_2$) is a line bundle on $\cT'$ (resp. on $\cU$) and $\varphi:M_1|_{\cU'}\simeq M_2|_{\cU'}$ is an isomorphism of line bundles.  An isomorphism between two line bundles is the multiplication by a scalar on each fiber, hence we can think of $\varphi$ as an element of $\Gm(\cU)=\Gm(\spec(K(R')))$, i.e. $\varphi=\pi^du$ where $d$ is an integer, $\pi$ is the uniformizer in $R$ which is also a uniformizer in $R'$ as $R\to R'$ is \'etale, and $u$ is invertible in $R'$.

    The data of a map $g:\cT\to\cB\Gm$ corresponds to a line bundle $M$ on $\cT$, and its pullback $p^*g:\cP\to\cB\Gm$ corresponds to the triple $(M|_{\cT'},M|_\cU,\Id)$. The arrow $g$ makes the diagram (\ref{eq:BGm eff diag}) 2-commutative if and only if there exists an isomorphism of the triple $(M|_{\cT'},M|_\cU,\Id)$ with $(M_1,M_2,\varphi)$, i.e. if there exist isomorphisms $a:M|_\cT'\to M_1$, $b:M|_\cU \to M_2$ such that $b|_{\cU'}\circ\varphi\circ a|_{\cU'} = \Id$. 
    
    First observe that, from \Cref{pic_gerbe_depends_only_on_Pic(Z)}, there is a unique line bundle $M$ on $\cT$ such that $M_2$ is the pull-back of $M$. Similarly,  as the map $\Pic(\cT')\to \Pic(\cU')$ is an isomorphism from \Cref{pic_gerbe_depends_only_on_Pic(Z)}, also the line bundle $M_1$ is the pull back of $M$ to $\cT'$. In particular $M_1\cong M|_{\cT'}$ and $M_2\cong M|_\cU$. So our goal is to find $a\in \Gm(\cT)=\Gm(\spec(R'))$ and $b\in\Gm(\cU)=\Gm(K(R))$ such that $a\cdot \pi^d u\cdot b=1$.  We can pick $a= u^{-1}$ and $b=\pi^{-d}$ to obtain this equality, thus proving that the map $g:\cT\to\cB\Gm$ given by $M$ makes (\ref{eq:BGm eff diag}) commutative.

    For $n\geq 1$, we argue by induction: indeed, there is a cartesian diagram
    \[
    \begin{tikzcd}
        \cB\Gm^n \ar[r] \ar[d] & \cB\Gm \ar[d] \\
        \cB\Gm^{n-1} \ar[r] & \spec(k).
    \end{tikzcd}
    \]
    Both the right vertical arrow and the bottom horizontal arrow are effective by inductive hypothesis, and the left vertical arrow is effective because being effective is stable under base change (\Cref{lm:composition}). As the composition of effective arrows is effective by the same lemma, we obtain the desired conclusion.
\end{proof}

\begin{Prop}\label{prop:BG eff}
Let $G$ be an algebraic group fitting in the short exact sequence
\[1\longrightarrow \Gm^n\overset{i}{\longrightarrow} G\longrightarrow F\longrightarrow 1\]
where the normal subgroup $i(\Gm^n)$ is central and $F$ is finite. Then $\cB G\to\spec(k)$ is effective.
\end{Prop}

\begin{proof}
    Due to the fact that $\Gm^n$ is central in $G$, we have from \cite{giraud2020cohomologie}*{ Proposition III.3.3.1; Remarque
IV.4.2.10}, two long exact sequences
    \[
    \begin{tikzcd} \oH^1(\cT,\Gm^n)  \ar[r, "j_\cT"] \ar[d, "p_1"]&
        \oH^1(\cT,G) \ar[r, "i"] \ar[d, "p_2"] & \oH^1(\cT,F) \ar[d, "p_3"] \ar[r, "\delta_\cT"] & \oH^2(\cT,\Gm^n) \ar[d, "p_4"]\\ \oH^1(\cP, \Gm^n) \ar[r, "j_\cP"] &
        \oH^1(\cP, G) \ar[r, "\iota"] & \oH^1(\cP, F) \ar[r, "\delta_\cP"] & \oH^2(\cP, \Gm^n).
    \end{tikzcd}
    \]
We recall that, as $F$ might not be abelian, $\oH^1(\cdot, G)$ and $\oH^1(\cdot, F)$ might not be groups, but only pointed sets. None the less, there is a functorial action of $\oH^1(\cdot,\Gm^n)$ on $\oH^1(\cdot, G)$.  This action is compatible with the exact sequence above, in the sense that given $x\in \oH^1(\cdot,\Gm^n)$ and $y \in \oH^1(\cdot,G)$, if we denote by $x*y$ the action of $x$ on $y$, then $\iota(x*y)=\iota(y)$; see \cite{giraud2020cohomologie}*{ III.3.4.5}.

Observe now that:
\begin{itemize}
\item $p_1$ and $p_3$ are surjective as $\cB \Gm^n$ and $\cB F$ are effective, and
\item $p_4$ is injective.
\end{itemize} 
Indeed the first bullet point follows from \Cref{prop_Gm_effective_group},  and the fact that $F$ is finite, so $\cB F$ is separated hence effective from \Cref{prop:sep to eff}. For the second bullet point, observe that the we have a commutative diagram
    \[
    \begin{tikzcd}
    \oH^2(\cT,\Gm^n) \ar[r] \ar[d] & \oH^2(\cU,\Gm^n) \ar[d, "\simeq"] \\
    \oH^2(\cP, \Gm^n) \ar[r] & \oH^2(\cU,\Gm^n)
    \end{tikzcd}
    \]
    and the right vertical arrow is an isomorphism. The homomorphism $\oH^2(\cT,\Gm)\to \oH^2(\cU,\Gm)$ is injective: indeed, as $\cT$ is regular, we have that $\oH^2(\cT,\Gm)$ is isomorphic to the Brauer group, and the Brauer group of $\cT$ injects into the Brauer group of its generic point. This is true if $\cT$ is a  scheme from \cite{grothendieck1968groupe}*{Corollarie 1.8}, whereas in the case where $\cT$ is a DM stack it suffices to observe that the same argument of \cite[Proposition 3.1.3.3]{lieblich2008twisted} goes through. Then the top horizontal arrow in the last diagram is injective, hence also the left vertical arrow is injective. 

Now the surjectivity of $p_2$, which is equivalent to $\cB G$ being effective, follows from a diagram chase of the first diagram of pointed sets.  We report it below. 

Let $\alpha \in \oH^1(\cP,G)$. Since $p_3$ is surjective, there is $b\in \oH^1(\cT,F)$ such that $p_3(t)=\iota(\alpha)$. But $\delta_P(\iota(\alpha))=0 = p_4(\delta_\cT(b))$. Since $p_4$ is injective, $\delta_\cT(b)=0$ so $b=i(a)$ for $a\in \oH^1(\cT,G)$. Now the two elements $p_2(a)$ and $\alpha$ map to the same element in $\oH^1(\cP,F)$. In particular, there is $\xi\in \oH^1(\cP,\Gm^n)$ such that $\xi*p_2(a) = \alpha$. But $p_1$ is surjective, so there is $x\in \oH^1(\cT,\Gm^n)$ which maps to $x$. Hence $p_2(x*a)=\alpha$, as desired.
\end{proof}
We are ready to prove one of the two directions of our main result.
\begin{theorem}\label{teo quotient implies effective}
    Assume that $\cX$ admits a DM morphism $\phi:\cX\to \cB \Gm^n$, and a separated good moduli space $X$. Then the good moduli space map $\cX\to X$ is effective.
\end{theorem}
\begin{proof}
    Let $\cY\to \cX$ be the $\Gm^n$-torsor associated to $\phi$. Then $\cY\to \cX$ is an affine morphism, so it is also S-complete from \cite[Proposition 3.42]{AHH}. So the composition
    $\cY\to \cX \to X$ is S-complete since being S-complete is stable under compositions, and $\cX$ is S-complete as it admits a separated good moduli space. But then $\cY$ is separated, since a DM stack is S-complete if and only if it is separated. In particular, the morphism $\phi$ is separated. So also from \cite[\href{https://stacks.math.columbia.edu/tag/04YV}{Tag 04YV}]{stacks-project} the morphism $\cX\to X\times \cB \Gm^n$ is separated. In particular, it is effective from \Cref{prop:sep to eff}; since $\cB \Gm^n\to \spec(k)$ is effective also $X\times \cB \Gm^n\to X$ is effective from \Cref{lm:composition}. Since composition of effective morphisms is effective, also $\cX\to X$ is effective as desired.
\end{proof}
\begin{Def}\label{def effective group}
    We say that a group $G$ is effective if $\cB G\to\spec(k)$ is effective.
\end{Def}
\begin{Lemma}\label{lemma_tensoring_with_gerbe_is_effective_implies_effective}
    Let $\cX\to \cY $ a morphism of algebraic stacks, and let $\cZ\to \cY$ a separated morphism which is a gerbe. If $\cX\times_{\cY} \cZ\to \cZ$ is effective, then $\cX\to \cY$ is effective.
\end{Lemma}
We will use the following auxiliary lemma.
\begin{Lemma}\label{lemma_auxiliary_almost_garbage}
    Let $\pi:\cX\xrightarrow{} \cX'$ a gerbe over a DM stack $\cX'$, and let $f:\cX\xrightarrow{} \cY$ be a morphism. Assume that $\cX$ and $\cX'$ are separated. Then for every geometric point $p\in \cX$ one has $$\operatorname{Ker}(\Aut_\cX(p)\to \Aut_{\cX'}(\pi(p)))\subseteq \operatorname{Ker}(\Aut_\cX(p)\to \Aut_{\cY}(f(p)))$$
    if and only if there is a map $\phi:\cX'\to \cY$ such that $\phi\circ\pi$ and $f$ are isomorphic.
\end{Lemma}
\begin{proof}
If there is a map $\phi$ as above, the inclusion of the kernels is clear. We focus on the other direction.

    Consider $(\pi,f):\cX\to \cX'\times \cY$. As the composition $\cX\to \cX'\times \cY\to \cX'$ is separated, and by our conventions all the relative diagonals are separated, also $(\pi,f)$ is separated. From \cite{AOV}*{Theorem 3.1} the map $(\pi,f)$ admits a relative coarse moduli space $\cX^{r.c.}$. So one has a factorization $\cX\xrightarrow{\alpha}\cX^{r.c.}\xrightarrow{\beta}\cX'\times \cY$; let $p_1:\cX'\times\cY\to \cX'$ be the first projection. Observe that for every geometric point $p$ of $\cX$:
    \begin{enumerate}
    \item as $\alpha$ is a relative coarse moduli space, we can (and will) identify (via $\alpha$) the geometric points of $\cX$ and $\cX'$,
        \item $\Aut_\cX(p)\to \Aut_{\cX^{r.c.}}(\alpha(p))$ is surjective,
        
        \item $ \Aut_{\cX^{r.c.}}(p)\to \Aut_{\cX'}((p_1\circ\beta)(p))$ is injective,
        \item $\cX^{r.c.}\to \cX'$ is bijective on geometric points as $\cX\to \cX'$ and $\cX\to \cX^{r.c.}$ are such.
    \end{enumerate}
    Point (3) follows easily from the containment of the two kernels in the assumptions and point (2).

    In particular the map $\Aut_{\cX^{r.c.}}(p)\to \Aut_{\cX'}(p)$ is injective for every $p$, so from \cite{conrad2007arithmetic}*{Theorem 2.2.5} it is representable. We show that $\cX^{r.c.}\to \cX'$ is an isomorphism \'etale locally on $\cX'$.
    
    Consider $U\to \cX'$ an \'etale atlas which is a scheme. We can pull back all the maps constructed before:
    $$\cX\times_{\cX'}U\to \cX^{r.c.}\times_{\cX'}U\to U\times\cY\to U.$$
    We need to show that $\gamma:=(p_1\circ\beta)|_{\cX^{r.c.}\times_{\cX'}U}:\cX^{r.c.}\times_{\cX'}U\to U$ is an isomorphism.
    As $\cX^{r.c.}\to \cX'$ is representable, $\cX^{r.c.}\times_{\cX'}U$ is an algebraic space.
    Moreover, as $\cX\to \cX'$ is a gerbe, $\cX\times_{\cX'}U\to U$ is the coarse moduli space map. In particular, from the universal property of the coarse moduli space, the map $\cX\times_{\cX'}U\to \cX^{r.c.}\times_{\cX'}U$ factors as $$\cX\times_{\cX'}U\to U\xrightarrow{\delta}\cX^{r.c.}\times_{\cX'}U$$
    and $\gamma\circ\delta=\Id$ from the universal property of the coarse moduli space. So $\delta$ has a left inverse.
        
    Similarly, the construction of the \textit{relative} coarse moduli space commutes with \'etale base change from \cite{AOV}*{Theorem 3.1}, so $\cX^{r.c.}\times_{\cX'}U$ is the relative coarse moduli space of $\cX\times_{\cX'}U\to  U\times \cY$. We can use the universal property of the relative coarse moduli space, exactly as before we were using the universal property of the coarse moduli space. In particular, the diagram $\cX\times_{\cX'}U\to U \xrightarrow{\delta} \cX^{r.c.}\times_{\cX'}U\to  U\times\cY$ factors as 
    $$ \cX\times_{\cX'}U\to \cX^{r.c.}\times_{\cX'}U \xrightarrow{\gamma'} U \xrightarrow{\delta} \cX^{r.c.}\times_{\cX'}U\to  U\times\cY.$$
    Hence $\delta\circ\gamma' = \Id$. So $\delta$ has both a right inverse and a left inverse, so $\gamma = \gamma'$ and $\gamma$ and $\delta$ are isomorphisms.
\end{proof}

\begin{proof}[Proof of \Cref{lemma_tensoring_with_gerbe_is_effective_implies_effective}]As usual, we need to check that a morphism $\cT\to \cY$ which lifts to $\cP\to \cX$, lifts to $\cT\to\cX$. We can pull back $\cX\to \cY$ via $\cT\to \cY$, so we can and will assume that $\cY=\cT$. Recall that $\cT$ is separated.

As all the fiber products will be over $\cT$, we will omit the subscript $\cT$ in $\times_\cT$. First observe that $\cZ\to \spec(R)\text{ is a DM gerbe}$, as it is a composition of the two gerbes $\cZ\to \cY=\cT$ and $\cT\to \spec(R)$.

Hence, the morphism $\cP\to \cX$ induces $\cP\times\cZ\to \cX\times\cZ$. But $\cX\times \cZ\to \cZ$ is effective, so in  the diagram below the map $\alpha$ has a lift
$$\xymatrix{ & \cX\times \cZ\ar[dr]\ar[d]_\alpha & \\ \cP\times \cZ\ar[ur]\ar[r]\ar[dr] & \cZ\ar[dr] & \cX\ar[d]\\ & \cP\ar[ur]\ar[r] & \cT.}$$

Consider then the maps $f:\cZ\to \cX\times \cZ\to \cX$
and $\pi:\cZ\to \cT$. The map $p_2:\cP\times \cZ\to \cZ$ is surjective and for every geometric point $p\in \cP\times \cZ$ it induces an isomorphism $\Aut_{\cP\times \cZ}(p)\to \Aut_{\cZ}(p_2(b))$. Hence to understand
$f:\Aut_{\cZ}(\pi_2(p))\to \Aut_{\cX}(f(\pi_2(p)))$ and $\pi:\Aut_{\cZ}(\pi_2(p))\to \Aut_{\cT}(\pi(\pi_2(p)))$,
one can understands the corresponding maps for $\cP\to \cZ\to \cX\times \cZ \to \cX$ and $\cP\times \cZ\to \cP$. 
Now as the map $\cP\times \cZ\to \cX\times \cZ\to \cX$ factors as $\cP\times \cZ\to \cP\to \cZ$, the inclusion of the kernels in \Cref{lemma_auxiliary_almost_garbage} hold. In turn, this gives that also $\cZ\to \cX$ factors via $\cZ\to \cT\to \cX$.
\end{proof}

\section{Torsors, twisted tori and gerbes}\label{section torsors and twisted tori}
In this section we recall some well-known facts about torsors, twisted tori and gerbes that will be needed in the rest of the paper. In particular, along the way we prove that the classifying space of a twisted torus is isomorphic to a certain fiber product. We will use these results in \Cref{section when BG not eff}.
\subsection{Torsors}
We will be mostly interested in (\'etale) torsors over the spectrum of a field $L$, and $G$ will be a finite type group scheme over $\spec(L)$. We begin with the following:

\begin{Def}
    If $f:X\to Y$ is a $G$-torsor, the automorphisms of $f$ are the automorphisms of $X$ which commute with $f$ and the (left) $G$-action.
\end{Def}
\begin{EG}\label{example aut trivial G torsor}
    If $f:G\times Y\to Y$ is the trivial torsor, there is a bijection between $G$ and the automorphisms of $f$. This bijection sends $G\ni g\mapsto ((h,y)\mapsto (hg^{-1},y))$. This is clearly an isomorphism of $G\times Y$ which commutes with $f$, and it commutes with the action since $h_1(h_2g^{-1})=(h_1h_2)g^{-1}$: the multiplication in $G$ is associative, and our actions are left actions.
\end{EG}

Torsors can be described using descent. Let us first fix some notation.
\begin{Notation}
    In what follows we use use $G$ to denote a group scheme over $k$, we use $X$ for a $G$-torsor which gets trivialized by a Galois extension $k\subseteq k'$, and finally we use $\Gamma$ to denote the Galois group of the extension $k\subset k'$. 
\end{Notation}
In particular $\Gamma$ acts on $G_{k'}$, the action commutes with the group operation in $G_{k'}$, and there is a bijection between the set of $G$-torsors $X\to \spec(k)$ and $H^1(\Gamma, G_{k'})$ (see \cite[Chapter 1 Section 5]{Galois_cohom}); in the following remark we construct one of the two arrows. 

\begin{Remark}\label{remark from cocycle to t functions}
In the context of group cohomology, consider a cocycle 
    $$\Gamma\longrightarrow G_{k'},\quad \gamma\longmapsto g_\gamma.$$
The cocycle condition is that $g_{\gamma\delta}=g_{\gamma}(\gamma*g_{\delta})$ where $*$ is the action of $\Gamma$ on $G_{k'}$. By descent, this defines a $G$-torsor over $\spec(k)$, which we can understand as follows. 

The cocycle defines an action of $\Gamma$ on $G_{k'}$ via $(\gamma,x)\mapsto (\gamma*x)\cdot g^{-1}_\gamma$, where we denoted by $\cdot$ the group multiplication. The cocycle condition guarantees that this is indeed an action, namely $$x\mapsto (\gamma*x)\cdot g^{-1}_\gamma\mapsto \delta*((\gamma*x)\cdot g^{-1}_\gamma)\cdot g_{\delta}^{-1} = ((\delta\gamma)*x)\cdot (\delta*g^{-1}_\gamma)\cdot g_{\delta}^{-1} =((\delta\gamma)*x) \cdot g_{\delta\gamma}^{-1}.$$
By taking the quotient of $G_{k'}$ with respect to this action, we get the desired $G$-torsor.
For example, the trivial cocycle $\gamma \mapsto 1$ gives rise to the trivial $G$-torsor: the $G$-torsor given by taking the $\Gamma$-invariants in $G_{k'}$. 
\end{Remark}
\begin{Def}
    Given a cocycle $\Phi:\Gamma\to G_{k'}$, $\gamma\mapsto g_\gamma$, for every $\gamma\in \Gamma$ we have an isomorphism $G_{k'}\to G_{k'}$, $x\mapsto (\gamma\ast x)\cdot g_\gamma^{-1}$. We will call these isomorphisms the \textit{transition functions} of the cocycle $\Phi$.
\end{Def}
 The transition functions are the ones that define a ``twisted'' action of $\Gamma$ on $G_{k'}$; the quotient of $G_{k'}$ with respect to this action is the $G$-torsor associated to the cocycle.
\begin{EG}\label{example cocycle that gives the galois extension}
    Consider the trivial action of $\Gamma$ on itself, and consider the $\Gamma$-torsor associated to the cocycle $\Id:\Gamma\to \Gamma$, whose associated transition functions are $g \mapsto g\gamma^{-1}$. Then the resulting $\Gamma$-torsor over $\spec(k)$ is the quotient of $\Gamma$ by the action on itself defined by the transition functions: this quotient is exactly $\spec(k')\to \spec(k)$, with the $\Gamma$-action being the action of the Galois group.
\end{EG}
\begin{Lemma}\label{lemma automorphisms of a torsor}
    Assume that $f:X\to \spec(k)$ is a $G$-torsor, assume that $f$ is trivialized on the Galois extension $k\subseteq k'$ with Galois group $\Gamma$, and assume that $X$ is given by the cocycle $\Gamma\to G_{k'}$, $\gamma\mapsto g_\gamma$. Then there is a (non-canonical) bijection between $\Aut(X)$ and the set $$ \{a\in G_{k'}\text{ such that }\gamma * a = g_{\gamma}^{-1}\cdot a\cdot g_\gamma\}.$$
\end{Lemma}
\begin{proof}
    The automorphisms of $f$ are the automorphisms of $f_{k'}:X_{k'}\to \spec(k')$ which descend, i.e. which commute with the transition functions. As $X_{k'}$ is a trivial torsor, up to choosing an isomorphism $X_{k'}\to G_{k'}$, we can identify the automorphisms of $X_{k'}$ with $G_{k'}$, as in \Cref{example aut trivial G torsor}. Those which descend are the $\beta$'s which make the following diagram commutative for every $\gamma\in \Gamma$:
    $$\xymatrix{G_{k'}\ar[d]_{x\mapsto (\gamma \ast x)\cdot g_\gamma^{-1}} \ar[r]^\beta & G_{k'}\ar[d]^{x\mapsto (\gamma \ast x)\cdot g_\gamma^{-1}} \\ G_{k'}\ar[r]^\beta & G_{k'}.}$$
    As $\beta$ is given by left multiplication for $a\in G_{k'}$, it is now straightforward to see the desired bijection.
\end{proof}

\subsection{Twisted tori and gerbes}
\begin{Def}
   A twisted torus over $k$ is a group scheme $G$ which admits an isomorphism $G_{\overline{k}}\cong (\Gm^n)_{\overline{k}}$ for a certain $n$, where $\overline{k}$ is the algebraic closure of $k$.
\end{Def}
It is well known that if $G$ is a twisted torus, then there is a finite Galois extension $k'/k$ together with an isomorphism $G_{k'}\cong (\Gm^n)_{k'}$. We now recall \cite[Proposition 2.4]{Arithmetictoricvars}:
\begin{Prop}\label{prop from homs to twisted tori} Let $k\subseteq k'$ a Galois extension with Galois group $\Gamma$.
    Then there is a bijection between homomorphisms $\Gamma\to \GL_n(\mathbb{Z})$ and pairs $(T,f)$ where $T\to \spec(k)$ is a twisted torus and $f:T_{k'}\cong (\Gm^n)_{k'}$ is an isomorphism. So this induces a bijection between conjugacy classes of homomorphisms $\Gamma\to \GL_n(\mathbb{Z})$ and twisted tori $T\to \spec(k)$ such that $T_{k'}\cong (\Gm^n)_{k'}$.
\end{Prop}
 The argument essentially follows from the fact that twisted tori as above are parametrized by the group $H^1(\Gamma, \Aut((\Gm^n)_{k'}))$, and the action of $\Gamma$ on $\Aut((\Gm)_{k'})$ is trivial. So the cocycle condition $a_{\gamma \delta}=a_\gamma \gamma\ast a_\delta$ becomes $a_{\gamma\delta}=a_\gamma a_\delta$ (namely, $a$ is a homomorphism); and two cocycles $a$ and $a'$ are conjugated if and only $a$ is the composition of $a'$ with an inner automorphism. 
 
\begin{Notation}\label{notation twisted torus}Given an homomorphism $\phi:\Gamma\to \GL_n(\bZ)$, we denote by $T_\phi$ the twisted torus induced by $\phi$ as in \Cref{prop from homs to twisted tori}.
\end{Notation}
\begin{Lemma}\label{lemma points of nonsplit torus using descent}
    Let $T_\phi$ be a torus twisted by $\phi$. Then, for every scheme $S$ over $\spec(k)$, there is a bijection $$T_\phi(S)\longleftrightarrow \{x \in \Gm^n(S_{k'})\text{ such that }x =\phi(\gamma)(\gamma \ast x)\}.$$
\end{Lemma}
\begin{proof}
 This follows from descent. Indeed, a morphism $S\to T_\phi$ is equivalent, via descent, to a morphism $S_{k'}\to \Gm^n$ which satisfies the descent condition. This is a morphism $x\in \Gm^n(S_{k'})$ which makes the following diagram commutative
 $$\xymatrix{S_{k'}\ar[d]^{\Id} \ar[r]^x & \Gm^n\ar[d]^{g\mapsto \phi(\gamma)(\gamma\ast g)} \\S_{k'}\ar[r]^x &\Gm^n.} $$
\end{proof}

Now, given a homomorphism $\phi:\Gamma\to \GL_n(\bZ)$ one can perform another construction, the semidirect product $\Gm^n\rtimes_\phi \Gamma$. The following result clarifies the relation between the group $\Gm^n\rtimes_\phi \Gamma$ and $\cB T_{\phi}$.
\begin{Teo}\label{theorem from semidirect product to nonsplit torus}
    Let $\phi:\Gamma \to \GL_n(\bZ)$ be a homomorphism, let $G:=\Gm^n\rtimes_\phi \Gamma$, and let $\spec(k)\to \cB\Gamma$ be the morphism associated to the $\Gamma$-torsor $\spec(k')\to \spec(k)$. Consider the following fiber diagram, where the map $\cB G\to \cB\Gamma$ is associated to the surjection $G\to \Gamma$, $(g,\gamma)\mapsto \gamma$:
    $$\xymatrix{\cG = \spec(k)\times_{\cB\Gamma}\cB G\ar[r] \ar[d] & \cB G\ar[d] \\\spec(k)\ar[r] &\cB \Gamma.}$$
    Then $\cG = \cB T_\phi$.
\end{Teo}
\begin{proof}
    As $\cB  G\to\cB\Gamma$ is a gerbe, we already know that $\cG\to\spec(k)$ is a gerbe. It suffices to prove that, if we denote with $\pi:\cG\to \spec(k)$ the first projection, then $\pi$ has a section $\xi$, and that there is an isomorphism $\underline{\Aut}_k(\xi)\cong T_\phi$: indeed, from the first statement it would follow that $\cG\simeq \cB \underline{\Aut}_k(\xi)$, which combined with the second statement would give us the desired conclusion.

    \underline{The map $\pi$ has a section:} from the universal property of the fiber product, it suffices to construct a $G$-torsor $\cF$ over $\spec(k)$ whose associated $\Gamma$-torsor $\cF/\Gm^n$ is isomorphic to $\spec(k')$: the latter statement is equivalent to proving that (1) the pullback to $\spec(k')$ of $\cF/\Gm^n$ is trivial and (2) the associated cocycle is the one in \Cref{example cocycle that gives the galois extension}.

    First observe that there is an action of $\Gamma$ on $(\Gm^n)_{k'} = \spec(k'[t_1^{\pm},...,t_n^{\pm}])$, which we denote by $\ast_{\Gm^n}$, as there is an action of $\Gamma$ on $\spec(k')$ (the Galois action). We also use $\ast_{\Gamma}$ to indicate the trivial action of $\Gamma$ on itself, so that $\sigma\ast_\Gamma \gamma :=\gamma$. This gives an action of $\Gamma$ on $G= \Gm^n\rtimes_\phi \Gamma$, via $$\sigma *(g,\gamma):=(\sigma\ast_{\Gm^n}g,\gamma).$$
    As $\ast_{\Gm^n}$ acts only on the coefficients of $(\Gm^n)_{k'} = \spec(k'[t_1^{\pm},...,t_n^{\pm}])$, for every $\sigma$ and $\gamma\in \Gamma$, the operation $\sigma \ast_{\Gm^n}$ commutes 
    with $\phi(\gamma)$, so $\ast$ is an action of $\Gamma$ on $G$ that commutes with the multiplication on $G$. So now we have an action of $\Gamma$ on $G$, and we can check that $$\Gamma\longrightarrow G,\quad \gamma\longmapsto a_\gamma = (1,\gamma)$$
    satisfies the cocycle condition. In particular, it induces a $G$-torsor over $\spec(k)$, which we denote by $\cF\to \spec(k)$. This induces a map $\spec(k)\to \cB G$.

    The $\Gamma$-torsor associated to $\cF$ is obtained as follows. By construction, $\cF_{k'}$ is the trivial $G$-torsor, so if we fix an isomorphism between $\cF_{k'}\to G_{k'}$ we can identify the source with the target. The trivial torsor $G_{k'}$ has a (left) action of $G_{k'}$, so one can take its quotient by $\Gm^n\subseteq G$, to get $\Gm^n\backslash G_{k'}$. From \Cref{remark from cocycle to t functions}, the cocycle for $\cF$ (namely, the cocycle
    $a_\gamma$) descends for a cocycle for $\Gm^n\backslash G_{k'}$ (namely, the transition functions commute with taking the $\Gm^n$-quotient). This is clear from \Cref{remark from cocycle to t functions} as $G_{k'}$ is associative: the elements $a_\gamma^{-1}$ are multiplied on the right, whereas the action is a left action.
    
    So $a_\gamma$ descend and give transition functions on $\Gm^n\backslash G_{k'}$. If we fix the isomorphism $\Gamma \to \Gm^n\backslash G_{k'}$, $\gamma \mapsto [(1,\gamma)]$ so that we can identify the quotient $\Gm^n\backslash G_{k'}$ with $\Gamma$, one can see that the cocycle $a_\gamma$ is the same as the one in \Cref{example cocycle that gives the galois extension}. In particular, $\Gm^n\backslash\cF$ is isomorphic to the $\Gamma$-torsor $\spec(k')\to \spec(k)$. So from the universal property of the fiber product, $\pi$ has a section $\xi:\spec(k)\to\cG$.

    \underline{$\underline{\Aut}_k(\xi)\cong T_\phi$:} from \Cref{lemma automorphisms of a torsor} we have
    \[\Aut_k(\cF) = \{a = (\lambda,\mu)\in (G)_{k'}\text{ such that }\gamma * a = g_{\gamma}^{-1}\cdot a\cdot g_\gamma\},\]
    and from the universal property of the fiber product the automorphisms of $\xi$ are the automorphisms of $\cF$ of the form $(\lambda, 1)$. Since conjugating with $g_{\gamma^{-1}}$ acts as $\phi(\gamma^{-1}) = \phi^{-1}(\gamma)$, we deduce 
\begin{align*}
    \Aut_k(\cF)
    =&\{(\lambda,1)\in (G)_{k'}\text{ such that }\gamma * a = g_{\gamma}^{-1}\cdot a\cdot g_\gamma\} \\
    =&\{\lambda \in (\Gm)_{k'} \text{ such that }\gamma * \lambda = \phi^{-1}(\gamma)(\lambda)\} \\
    =&\{\lambda \in (\Gm)_{k'} \text{ such that }\phi(\gamma)(\gamma * \lambda) = \lambda\} \\
    =&T_\phi(\spec(k))
\end{align*}
where for the last bijection we used \Cref{lemma points of nonsplit torus using descent}.
\end{proof}

\section{When $\cB G$ is not effective}\label{section when BG not eff}

In this section we investigate when the classifying space of a group $G$ is not effective. We begin with the following criterion, which will be our main tool to prove that $\cB G$ is not effective.

\begin{theorem}\label{teo criteria for BG to be effective with double quotient}
    Let $G$ be a connected reductive group, let $b:\spec(R')\to \spec(R)$ be an \'etale morphism of DVRs, let $\alpha:\spec(K(R'))\to \spec(R')$ and $a: \spec(K(R))\to \spec(R)$ be the inclusions of the generic points, and let $\beta:\spec(K(R'))\to \spec(K(R))$ be the map between the generic points. Let $\cG_1\to \spec(K(R))$ and $\cG_2\to \spec(R')$ be two $G$-torsors, which are isomorphic when pulled back to $\spec(K(R'))$. If $\cB G$ is effective, then the double quotient 
    $$\Aut(\cG_1)\backslash \operatorname{Isom}(\beta^*\cG_1,\alpha^*\cG_2)/\Aut(\cG_2)$$
    is a single element.
\end{theorem}
\begin{proof}
    Let $P:=P_{R,R'}$. A $G$-torsor over $P$ consists of two $G$-torsors over $\spec(K(R))$ and $\spec(R')$, and an isomorphism between their restrictions to $\spec(K(R'))$.
We will argue by contradiction: let $\psi$ and $\widetilde{\psi}$ be two elements in $$\operatorname{Isom}(\cG_1|_{ K(R')},\cG_2|_{K(R')})$$ which are not in the same equivalence class. Let $\cG_P$ be the torsor
associated to $\cG_1\to \spec(K(R))$, $\cG_2\to \spec(R')$ and $\psi$; and $\widetilde{\cG}_P$ the one associated to $\cG_1\to \spec(K(R))$, $\cG_2\to \spec(R')$ and $\widetilde{\psi}$.

First, observe that $\cG_P$ and $\widetilde{\cG}_P$ are not isomorphic. Indeed, an isomorphism between $\cG_P$ and $\widetilde{\cG}_P$ consists of two elements $\sigma\in \Aut(\cG_1)$ an $\tau \in \Aut(\cG_2)$ such that $\psi = \sigma\circ \widetilde{\psi}\circ \tau$. Since $\psi$ and $\widetilde{\psi}$ are not in the same equivalence class, such an isomorphism cannot exist. 

Then notice that $\cG_P$ and $\widetilde{\cG}_P$ restrict to the same $G$-torsor over $\spec(K(R))$, namely $\cG_1$.
But then at least one between $\cG_P$ and $\widetilde{\cG}_P$ is not the pull-back of a $G$-torsor over $\spec(R)$. Indeed, if there were two $G$-torsors $\cG_R$ and $\widetilde{\cG}_R$ which pull back to $\cG_P$ and $\widetilde{\cG}_P$ respectively, then
\begin{enumerate}
    \item $\cG_R$ and $\widetilde{\cG}_R$ would not be isomorphic as $\cG_P$ and $\widetilde{\cG}_P$ are not isomorphic, but
    \item They would be isomorphic when restricted to $P$ and hence to $\spec(K(R))$.
\end{enumerate}
This contradicts \cite{grothendieck_Serre_panin}*{Corollary 1}. In particular $\cB G\to\spec(k)$ cannot be effective.
\end{proof}

\begin{Notation}\label{notation_fields_L_and_Lprime}
    In this section, $L\subseteq L'$ will be a Galois extension of fields with Galois group $\Gamma$. Then the extension $L(t)\subseteq L'(t)$ is still Galois with Galois group $\Gamma$, and $\Gamma$ acts trivially on $t$. In particular, it induces an \'etale extension of DVRs $\spec(L'[t]_{(t)})\to  \spec(L[t]_{(t)})$, we denote by $R$ (resp. $R'$) the ring $L[t]_{(t)}$ (resp. $L'[t]_{(t)}$). 
\end{Notation}

\begin{Prop}\label{prop sl2 not effective}
The morphism $\cB \SL_2\to \spec(k)$ is not effective.
\end{Prop}
\begin{proof}Let $R:=L[t]_{(t)}$ and $R':=L'[t]_{(t)}$.
Pick an element $u\in L'\smallsetminus L$. Since all the $\SL_2$-torsors on a local ring are trivial (as $\SL_2$ is special), from \Cref{teo criteria for BG to be effective with double quotient} to prove the proposition it suffices to prove that the double quotient
$$\SL_{2,R'}\backslash \SL_{2,K(R')}/\SL_{2,K(R)}$$
is not trivial, where $\SL_{2,R'}$ acts by multiplication on the left and $\SL_{2,K(R)}$ by multiplication on the right. In particular, if $t$ is an uniformizer for $R$, it suffices to check that the following matrix is not in the orbit of the identity:
$$M :=\begin{bmatrix}t & 0\\u & t^{-1}\end{bmatrix}.$$
If it was in the orbit of the identity, we could find a matrix $$B:= \begin{bmatrix}
    a & b \\ c & d
\end{bmatrix}\in\SL_{2,K(R)} \text{ such that }MB =\begin{bmatrix}t a & t b\\ ua + t^{-1}c& ub + t^{-1}d\end{bmatrix} \in\SL_{2,R'}.$$

Since $ua + t^{-1}c\in R'$, we have $t(ua + t^{-1}c)\in R'$ so either $c=0$ or $v(c)\ge 0$: either way $c\in R'$.
Observe also that one between $a$ and $b$ does not belong to $R$, otherwise $MB$ would not have determinant one, as the first row of $MB$ is a multiple of $t$. So let's assume $a\notin R$, so $v(a)=-1$ as $ta\in R'$, where $v$ is the valuation on $R$. So we can write $a = t^{-1}a_1$ with $a_1$ unit in $R$, and $t^{-1}(ua_1 + c)\in R$. 
Therefore 
$$v(ua_1 + c)> 0.$$
If we consider $ua_1 + c$ in $L'=R'/(t)$, we have that the image of $u$ via $R'\to L'$ belongs to $L$ (it is equal to $-\frac{c}{a_1}$ modulo $t$), which is a contradiction.
\end{proof}
\begin{Prop}\label{prop nonsplit tori not eff} Let $\phi:\Gamma\to \GL_n(\bZ)$ be a homomorphism, and let $T_\phi\to \spec(L)$ be its associated non-split torus as in \Cref{notation twisted torus}. Then $\cB T_\phi$ is not effective.
\end{Prop}
\begin{proof} 
We plan to apply \Cref{teo criteria for BG to be effective with double quotient}, using \Cref{lemma points of nonsplit torus using descent} to show that not all the elements of the form $(1,...,1,t,1,...,1)$ are in the same equivalence class of the identity in 
$$(\Gm^n)_{R'}\backslash (\Gm^n)_{L'(t)}/\Aut(T_\phi(L(t)).$$

Set $\cG_1:=T_{\phi,R}\to \spec(R)$, $\cG_2:={\Gm^n}_{,R'}\to \spec(R')$ and $f:(\cG_1)|_{L'(t)}\overset{\simeq}{\to} (\Gm)_{L'(t)}^n$ an isomorphism. We will use $f$ to identify $$(\Gm^n)_{L'(t)}\simeq \Isom_{\Gm^n\text{-torsors}}((\Gm^n)_{L'(t)},(\Gm^n)_{L'(t)}) \overset{\simeq}{\longrightarrow}\operatorname{Isom}((\cG_1)|_{L'(t)},(\cG_2)|_{L'(t)}).$$
There is a valuation $v:(\Gm)_{L'(t)}\to \mathbb{Z}$, which sends $p(t)$ to its valuation at $0\in \bA^1_{L'}$. By definition, the valuation is a group homomorphism from the multiplicative group $(\Gm)_{L'(t)}$ to the additive group $\bZ$. In this way we get a homomorphism
$$v^n:(\Gm^n)_{L'(t)}\longrightarrow \mathbb{Z}^n.$$
Moreover, for every $n\times n$ matrix $M$ with integer coefficients, one has a homomorphism $F_M:(\Gm^n)_{L'(t)}\to (\Gm^n)_{L'(t)}$ defined as
\[(a_{i,j})\cdot (\lambda_1,...,\lambda_n)=((\lambda_1^{a_{1,1}}\lambda_2^{a_{1,2}}\cdots\lambda_n^{a_{1,n}}),\ldots,(\lambda_1^{a_{n,1}}\lambda_2^{a_{n,2}}\cdots \lambda_n^{a_{n,n}})).\]
Similarly, every matrix $M = (a_{i,j})$ gives a homomorphism $G_M:\bZ^n\to \bZ^n$ by matrix multiplication. It is straightforward to check that $v^n\circ F_M=G_M\circ v^n$. In particular, every element $\phi(\gamma)$ induces a homomorphism $G_{\phi(\gamma)}$, and $\phi(\gamma) = \Id$ if and only if $G_{\phi(\gamma)} = \Id$. 
Moreover, as $\Gamma$ acts trivially on $t$, for every $\gamma\in \Gamma$ one has $$v^n(\gamma \ast \lambda)= v^n(\lambda)$$ where we denoted by $\ast$ the action on $(\Gm^n)_{L'(t)}$ induced by the Galois action of $\Gamma$ on $L'$. In particular, using \Cref{lemma points of nonsplit torus using descent}, the valuation $v^n$ sends the $L$-points of $(T_\phi)_{L'(t)}$ to the vectors in $\bZ^n$ that commute with $G_{\phi(\gamma)}$ for every $\gamma \in \Gamma$. Similarly, it sends all the elements in $(\Gm^n)_{R'}$ to 0. Thus, since
\begin{enumerate}
    \item we can identify the automorphisms of the trivial $T_\phi$-torsor with $T_\phi$ as in \Cref{example aut trivial G torsor},
    \item we have the bijection of \Cref{lemma points of nonsplit torus using descent}, and
    \item $v^n(\gamma \ast g)= v^n(g)$,
\end{enumerate}
we have a surjective map
$$(\Gm^n)_{\spec(R')}\backslash (\Gm^n)_{\spec(L'(t))}/\Aut(T_\phi(\spec(L(t))))\to \bZ^n/\{w\in \bZ^n \text{ such that }G_{\phi(\gamma)}(w) = w \text{ for every }\gamma \in \Gamma\}.$$
As there is a $\gamma$ such that $\phi(\gamma)$ is not trivial, there is a $\gamma$ such that $G_{\phi(\gamma)}$ is not trivial. In particular, it cannot fix all the elements in the standard basis of $\bZ^n$: this implies that the right hand side is not a single element, hence the left hand side cannot be a single element.
\end{proof}
\begin{Cor}\label{cor semidirect product is not effective unless phi is trivial}
    Assume that $F$ is a finite group, and let $\phi:F\to \GL_n(\bZ)$ be a homomorphism. Then  $\cB \Gm^n\rtimes_\phi F$ is effective if and only if $\phi$ is trivial.
\end{Cor}
\begin{proof} If $\phi$ is trivial, then $\cB (\Gm^n\times F)$ is effective from \Cref{prop:BG eff}, so let's assume that $\phi$ is not trivial, let $G:=\Gm^n\rtimes_\phi F$ and consider $L\subset L'$ a field extension with Galois group $F$ (which exists, for example, if $L=\spec(k(t))$). From \Cref{theorem from semidirect product to nonsplit torus} and \Cref{prop nonsplit tori not eff}, the map $\cB G\times_{\cB F}\spec(L)\to \spec(L)$ is not effective. Then from \Cref{lm:composition} the stack $\cB G$ cannot be effective.
\end{proof}
We are ready to prove the main result of this section.
\begin{Teo}\label{Teo BG effective implies G is a central ext of a torus}
    Assume that $G$ is a group such that $\cB G$ is effective. Then the connected component of the identity of $G$ is a split torus contained in the center of $G$.
\end{Teo}
\begin{proof}First we show that the connected component of the identity of $G$, which we denote by $H$, is a torus. Observe that the map $\cB H\to \cB G$ is separated, hence effective from \Cref{prop:sep to eff}. Therefore if $\cB G$ was effective, also $\cB H$ would be effective.

Since $G$ is reductive by our conventions,
and $H$ is a normal subgroup of $G$, then $H$ is reductive \cite{Alp21}*{Proposition 6.3.17}. Every reductive connected group which is not a torus admits a homomorphism $\SL_2\to H$ with finite kernel. Indeed, up to a finite base change, we can assume that $H$ is split reductive: by picking a root $\alpha$ of $H$, we can consider the associated split reductive subgroup $H_\alpha\subset H$ of rank one \cite{Mil}*{Theorem 21.11}; there is then a homomorphism $SL_2\to H_\alpha$ which is an isogeny onto the derived subgroup of $H_\alpha$ \cite{Mil}*{Proposition 20.32}, and the composition $\SL_2\to H_\alpha\to H$ has the claimed property. 

Hence we have a separated morphism $\cB\SL_2\to \cB H$. Therefore if $\cB H\to \spec(k)$ was effective, also $\cB\SL_2\to \spec(k)$ would be effective. \Cref{prop sl2 not effective} gives the desired contradiction.

So $H$ is a torus, and from \Cref{prop nonsplit tori not eff} it has to be split. From \cite{brionext}, if we denote by $F:=G/H$, we have a finite and surjective morphism $f:H\rtimes_\phi F\to G$. So $\cB(H\rtimes_\phi F)\to \cB G$ is effective, thus $\cB(H\rtimes_\phi F)$ is effective. It follows then that $\phi$ is trivial from \Cref{cor semidirect product is not effective unless phi is trivial}, and $H$ is contained in the center of $H\rtimes F$; as $f$ is surjective $f(H)$ is contained in the center of $G$ as desired.
\end{proof}
\begin{Cor}\label{cor:stab are extensions of tori}
    Assume that $\cX$ is an algebraic stack over $\spec(k)$ with a separated good moduli space $X$, and assume that $\cX\to X$ is effective. Then the stabilizers of the closed points are central extensions of split tori by finite groups.
\end{Cor}
\begin{proof}
This follows as the inclusion of the residual gerbe is a closed embedding (hence effective), composition of effective morphisms is effective, and \Cref{Teo BG effective implies G is a central ext of a torus}.
\end{proof}
\begin{Remark}\label{remark_only_fields_needed_to_check_BG_effective}
    Observe that in \Cref{prop nonsplit tori not eff} and \Cref{prop sl2 not effective}, to prove that $\SL_2$ is not effective and that $\Gm^n\rtimes_\phi F$ is effective if and only if $\phi=\Id$, we just used extensions as the ones in \Cref{notation_fields_L_and_Lprime}, and the trivial gerbe $\cT\to \spec(L[t]_{(t)})$. Hence, to show that $\cB G$ is effective, it suffices to prove that given an extension of DVRs as in \Cref{notation_fields_L_and_Lprime}, with local bug-eyed cover $P$, any morphism $P\to \cB G$ descends to $\spec(R)\to \cB G$.
\end{Remark}

\section{From effective to the quotient presentation}\label{section from eff to quotient}
In this section we prove the other direction of \Cref{teo intro}, namely that if a good moduli space morphism $\cX\to X$ is effective, then $X$ is a quotient of a DM stack by a torus (see \Cref{teo effettivo implica ho il gen set}). After introducing a general definition that will be useful later, we will proceed with the main result in two steps. First, in \Cref{ssection gms of dim 0} we will prove the desired result in the case $X=\spec(L)$ is the spectrum of a field. Then, in \Cref{ss extending gset} we prove the general result by spreading out line bundles and using \Cref{prop passo induttivo generating set}.  
\begin{Def}\label{def:generating line bundles}
    We say that an algebraic stack $\cX$ admits $n$ \textit{generating line bundles} if there are $n$ line bundles on $\cX$ such that the corresponding map $\cX\to \cB \Gm^n$ is representable in DM stacks.
\end{Def}
The following lemma motivates the name above.
\begin{Lemma}\label{lm:generating generate}
    Let $G$ be a group and assume that $\cB G$ has a generating set of line bundles. Then any such set form a generating set for the rational Picard group $\Pic(\cB G)_\bQ$; viceversa any set of line bundles whose classes generate $\Pic(\cB G)_{\bQ}$ also form a generating set on $\cB G$ in the sense of \Cref{def:generating line bundles}.
\end{Lemma}
\begin{proof}
    Let $\cL_1,\ldots,\cL_n$ be a generating set in the sense of \Cref{def:generating line bundles}. Identifying $\Pic(\cB G)_\bQ$ with the vector space of rational characters of $G$, let $\chi_1,\ldots,\chi_n$ be the associated characters, and let $\phi:G\to\Gm^n$ be the associated homomorphism of groups. Observe that $\phi$ induces precisely the DM morphism $\cB G\to\cB \Gm^n$, and that the torsor associated to $\cL_1\oplus\cdots\oplus \cL_n$ is $[(\bA^1\smallsetminus\{0\})^n/G]$, where the action of $G$ is the defined by $\phi$; the fact that this quotient stack is DM is equivalent to saying that the kernel of $\phi$ is finite: let us denote it $F$, and say that it has order $e$.

    Given any character $\chi$ of $G$, we have that $\chi^e(F)=0$, hence the character $\chi^e:G\to\Gm$ factors as $G\overset{\phi}{\to}\Gm^n\to\Gm$. This implies that $\phi^*$ induces a surjective morphism between the spaces of rational characters, and that $\chi_1,\ldots\chi_n$ generate $X(G)_\bQ.$ In terms of classifying stacks, this amounts to saying that $\cL_1,\ldots,\cL_n$ generate $\Pic(\cB G)_\bQ$.

    Now let $\cF_1,\ldots,\cF_m$ be generators for $\Pic(\cB G)_\bQ$, and let $\eta_1,\ldots\eta_n$ be the associated characters, so that there exist coefficients $a_{ij}\in \bQ$ with $\chi_i=\sum a_{ij}\eta_j$. Define $\psi:G\to\Gm^m$ as $g\mapsto (\eta_1(g),\ldots,\eta_m(g))$. There exists an integer $d$ such that the matrix $dA=(da_{ij})$ has integer coefficients and it induces a homomorphism $\ell_{dA}:\Gm^m\to\Gm^n$ having the property that the composition $\ell_{dA}\circ \psi$ is equal to $\phi^d$. This implies, as the kernel of $\phi$ is finite, that also the kernel of $\psi$ must be finite, hence the torsor associated to $\cF_1\oplus\cdots\oplus\cF_m$ is a DM stack, i.e. these line bundles form a generating set in the sense of \Cref{def:generating line bundles}. 
\end{proof}
\subsection{Base step: effective good moduli spaces of dimension zero}\label{ssection gms of dim 0}
In this subsection we prove one direction of \Cref{teo intro}, in the case where the good moduli space of $\cX$ is a point.
\begin{Lemma}\label{lemma la conj holds for gerbes}
    Assume that $f:\cX\to \spec(M)$ is a gerbe, and it is effective. Then there are $n$ generating line bundles on $\cX$.
\end{Lemma}

Part of the following argument was suggested by Siddarth Mathur and Minseon Shin to the second author.
\begin{proof}
Observe that $f$ is a good moduli space. Indeed, one can check that a morphism is a good moduli space \'etale locally on the target, so let $\spec(L')\to \spec(L)$ be an extension such that $\cX|_{L'}$ has a section.
By our conventions the stabilizers are reductive, so up to possibly extending $L'$ further, we have that the automorphisms of $\cX|_{L'}$ are a reductive group. Then $\cX|_{L'}\cong \cB G$ for a reductive group $G\to \spec(L')$, and $\cB G\to \spec(L')$ is a good moduli space.

Recall that while the gerbe $\cX\to \spec(M)$ might not be banded, there is a group object $\cZ\to \spec(M)$ satisfying the following two conditions:
\begin{enumerate}
    \item there is an isomorphism $\psi_x:\cZ|_T \to Z(\Aut_T(x))$, for every $x:T\to \cX$, where $Z(\Aut_T(x))$ denotes the center of $\Aut_T(x)$, and
    \item the morphism $\psi_x$ is canonical. Namely, for every isomorphism $\alpha:x\to y$ over $T$, if we denote by $\operatorname{Inn}_\alpha:\Aut_T(x)\to \Aut_T(x) $, $\sigma\mapsto \alpha\sigma \alpha^{-1}$, we have $\psi_y=\operatorname{Inn}_\alpha\circ \psi_x$.
\end{enumerate}
In particular, if $\Aut_T(x)$ is abelian, the gerbe $\cX\to \spec(M)$ is banded (the proof is analogous to \cite[\href{https://stacks.math.columbia.edu/tag/0CJY}{Tag 0CJY}]{stacks-project}).

Let $\cG\subseteq \cZ$ be the connected component of the identity.  For every map $x:\spec(M)\to \cX$ where $M$ is a field, we have that $\cB \Aut_{\cX}(x) \to \spec(M)$ is effective. Indeed, the map $\cB \Aut_{\cX}(x) \to \spec(M)$ is the second projection of $\cX\times_{\spec(L)}\spec(M)\to \spec(M)$, and being effective is stable under base change.

This implies by the characterization of effective classifying stacks (\Cref{Teo BG effective implies G is a central ext of a torus}) that the connected component of the identity of $\Aut_{\spec(M')}(x)$ is a split torus, contained in the center.
Hence $\cG|_{\spec(M')}$ coincides with this connected component and it is a split torus; from \cite{Con14}*{3.1.8} we deduce that also $\cG$ is a torus. We now show that $\cG$ is a split torus.

    We begin by rigidifying $f$ (see \cite{Alp21}*{Section 6.2.8}), so let $\cX':=\cX\sslash\cG$.
    We have a factorization $\cX\xrightarrow{h} \cX'\xrightarrow{g} \spec(M)$ where the diagonal relative to $h$ is represented by torsors under $\cG|_{\cX'}$, and $\cX'$ is DM over $\spec(M)$. As $f$ is effective, by \Cref{cor composition effective implies the first one effective} the morphism $h$ is also effective. The morphism $h$ is banded by $\cG$, let $c\in \oH^2(\cX',\cG)$ be the corresponding element.
    
    We will prove that $c$ is torsion using \cite[\href{https://stacks.math.columbia.edu/tag/03SH}{Tag 03SH}]{stacks-project}.
    Indeed, let $\spec(M)\to \spec(L)$ be an extension such that $\cG|_{\spec(M)}$ becomes a split torus. Then $\cX|_{\spec(M)}\to \cX'|_{\spec(M)}$ is a gerbe banded by $\Gm^n$ for a certain $m$. As $\cX'$ is a DM gerbe over the spectrum of a field, from \cite[Theorem 4.5.1]{Alp21} there exists a
    finite flat cover $\phi:U\to \cX'|_{\spec(M)}$ of some degree $d$ where $U$ is a regular scheme. The composition
    $$\oH^2(\cX',\cG|_{\cX'})\xrightarrow{\phi^*} \oH^2(U,\cG|_U)\cong \oH^2(U,\Gm^m)\xrightarrow{\phi_*}\oH^2(\cX',\cG)$$
    is the multiplication by $d$, so $c\in \oH^2(\cX',\cG|_{\cX'})$ is torsion since $\oH^2(U,\Gm^m)$ is a torsion group, as it agrees with the Brauer group of the regular scheme $U$. Let $k$ be the order of $c$.

    Observe now that, as $\cG$ is abelian, the multiplication by $k$ is an homomorphism; it is surjective with finite kernel as \'etale locally it is such. Specifically, there is an exact sequence $$1\to F\to \cG\xrightarrow{x\mapsto x^k} \cG\to 1$$ 
    on the \'etale site of $\spec(M)$, and $F\to \spec(M)$ is a \textit{finite} group-object over $\spec(M)$, which gives $$\oH^2(\cX',F|_{\cX'})\to\oH^2(\cX',\cG|_{\cX'})\xrightarrow{x\mapsto x^k} \oH^2(\cX',\cG|_{\cX'}). $$
    In particular, there is a gerbe $\cF\to \cX'$ banded by $F|_{\cX'}$ whose image via $\oH^2(\cX',F|_{\cX'})\to\oH^2(\cX',\cG|_{\cX'})$ is $\cX$. More precisely, using \cite[IV.3.1.8]{giraud2020cohomologie}, there is a gerbe $\cF\to \cX'$ banded by $F|_{\cX'}$ such that $$\cX\cong (\cF\times_{\cX'}\cB \cG|_{\cX'})\sslash (F|_{\cX'})\cong (\cF\times_{\spec(M)}\cB \cG)\sslash (F|_{\cX'}).$$
    In particular there is a map $\cF\times_{\spec(M)} \cB \cG\to \cX$ which is a gerbe banded by $F|_{\cX'}$. As $F$ is finite, this map is separated, hence effective. But then also $\cF\times_{\spec(M)}\cB\cG\to \spec(M)$ is effective, as composition of effective morphisms is effective. Therefore, from \Cref{lemma_tensoring_with_gerbe_is_effective_implies_effective} also $\cB\cG\to \spec(M)$ is effective: from \Cref{Teo BG effective implies G is a central ext of a torus} the group $\cG$ is a split torus.

    Therefore the class $c$ of $\cX\to \cX'$ is in $\oH^2(\cX',\Gm^m)$ for a certain $m$. As $c$ is torsion there exists $m$ line bundles on $\cX$ which are $n$-twisted, for $n$ divisible enough \cite{lieblich2008twisted}. This implies that the corresponding morphism $\cX\to \cX'\times \cB \Gm^m$ is relatively DM. As $\cX'\to \spec(L)$ is DM, the composition $\cX\to \cX'\times \cB \Gm^n\to \cB \Gm^n$ is relatively DM.
\end{proof}
\begin{Lemma}\label{lemma extend line bundles quotient}
    Let $B$ be a semi-local ring, and let $\cX\to \spec(B)$ be a good moduli space.  Then: 
    \begin{enumerate}
        \item the restriction map $i^*\colon \operatorname{Pic}(\cX)\to \bigoplus_{\overline{p}\to X} \operatorname{Pic}(\cB \cG_{\overline{p}})$ is injective, where the sum runs over all the closed geometric points $\overline{p}$ of $X$ and $\cB \cG_{\overline{p}}$ is the residual gerbe at $\overline{p}$;
        \item assume furthermore that $B$ is local with closed point $p$, that there exists a group $G\to\spec(k)$ which is a central extension of a split torus by a split finite group, and that the residual gerbe at the unique closed geometric point $\overline{p}$ of $\cX$ is $\cB G_{k(\overline{p})}$. Then if $\cX\cong [\spec(A)/G]$, the restriction map $i^*\colon \operatorname{Pic}(\cX)\to \operatorname{Pic}(\cB G_{k(\overline{p})})$ is an isomorphism.
    \end{enumerate}
    Moreover, with the assumptions of (2), the composition $\varphi:\operatorname{Pic}(\cX)\to \operatorname{Pic}(\cB G_{k(\overline{p})})\to \operatorname{Pic}(\cB G)$ is an isomorphism, where the last morphism is the one given by \Cref{lemma_pic_remains_the_same_after_field_ext}.
\end{Lemma}
\begin{proof}
(1) follows from \cite[Theorem 10.3]{alp}. Indeed, if $i^*\cL$ is trivial, then the action of $\cG_{k(\overline{p})}$ on the fiber of $\cL$ at $\overline{p}$ is trivial, so from \textit{loc. cit.} $\cL$ is pulled back from a line bundle on $\spec(B)$. As the latter is semi-local, every line bundle on $\spec(B)$ is trivial.

For (2), we already know injectivity from (1). To prove surjectivity, consider the composition $\cB G_{k(\overline{p})} \to [\spec(A)/G] \to \cB G$, where the latter morphism is the one induced by the $G$-torsor $\spec(A)\to [\spec(A)/G]$. Observe that this composition coincides with the morphism $\cB G_{k(\overline{p})}\to\cB G$ induced by the field extension $k\subset k(\overline{p})$. Therefore, we have that the composition of pullback maps
\[ \Pic(\cB G) \longrightarrow \Pic([\spec(A)/G]) \longrightarrow \Pic(\cB G_{k(\overline{p})}) \]
is an isomorphism because of \Cref{lemma_pic_remains_the_same_after_field_ext}. It follows then that the last map must be surjective.

The moreover part follows directly from point (2) and \Cref{lemma_pic_remains_the_same_after_field_ext}.
\end{proof}
\color{black}
\begin{Lemma}\label{remark local pic determided by res gerbe}
    With the assumptions and notations of \Cref{lemma extend line bundles quotient} and $B$ local, if $\cX$ admits a generating set of line bundles, then the morphism $\Pic(\cX)_{\bQ}\to\Pic(\cB \cG_{\overline{p}})_{\bQ}$ is an isomorphism. If $\cB \cG_{\overline{p}}$ is effective and $\cG_{\overline{p}}$ is split, then the vice versa hold.
\end{Lemma}
\begin{proof} 
Let $G:=\cG_{\overline{p}}$.
In one direction, we only have to prove surjectivity. Given a DM morphism $\cX\to\cB\Gm^n$, the composition $\cB G\to \cX \to \cB \Gm^n$ is also DM. In particular, there is a $\Gm^n$-torsor $\cP\to \cB G$ which is a DM stack. As shown in the proof of \Cref{lm:generating generate}, this datum is equivalent to the datum of a homomorphism of groups $\varphi:G\to\Gm^n$ having finite kernel $F$ and such that $\varphi^*:X(\Gm^n)_{\bQ}\to X(G/F)_{\bQ}$ is surjective, hence $\Pic(\cB \Gm^n)_{\bQ}\to \Pic(\cB G)_{\bQ}$ is surjective. As the latter map factors through $\Pic(\cX)_{\bQ}$, we deduce that $\Pic(\cX)_{\bQ}\to\Pic(\cB G)$ is surjective as well. 

In the other direction, suppose that $\Pic(\cX)_{\bQ}\to\Pic(\cB )_{\bQ}$ is surjective: then if there are generating line bundles on $\cB G$, they extend (up to taking high enough tensor powers) to line bundles over $\cX$, and the associated extended torsors will still be DM by upper-semicontinuity of the dimension of the automorphism groups of points. But we know that, thanks to \Cref{lemma la conj holds for gerbes}, for $\cB G$ effective the stack $\cB G$ has generating line bundles.
\end{proof}

\begin{Prop}\label{prop la conj holds for 0 dim gms}
    Assume that $\pi:\cX \to \spec(L)$ is an effective good moduli space map, and $L$ is a field. Then $\cX$ admits $n$ generating line bundles.
\end{Prop}
\begin{proof}
First consider the residual gerbe $j:\cG\to \cX$ at the closed point of $\cX$. Recall that $j$ is a closed embedding. Indeed, it is locally closed from \cite[Proposition 3.5.16]{Alp21}, and by passing to a local quotient presentation \cite{luna}, one can check that it corresponds to the inclusion of the \textit{closed} orbit.
So the composition $\pi\circ j$ is effective.
In particular, from \Cref{lemma la conj holds for gerbes}, there are generating line bundles $\cL_1,\ldots,\cL_n$ on the residual gerbe. We want to prove that they extend to line bundles $\cL_1,...,\cL_n$ on $\cX$.

From \cite{luna}, there is an \'{e}tale cover $\spec(L')\to \spec(L)$ and an isomorphism $[\spec(A)/G]\cong \cX\times_{\spec(L)}\spec(L')$, where $G$ is the automorphism group of the closed point of $\cX\times_{\spec(L)}\spec(L')$, and it is split; up to taking refinements of this cover, we can assume that $\spec(L')\to\spec(L)$ is Galois with Galois group $\Gamma$, and that the closed $G$-orbit of $\spec(A)$ has an $L'$-point. 

Observe moreover that as $\cB G \to [\spec(A)/G] \to \spec(L')$ is effective (the first morphism is the closed immersion of the closed orbit, the second is effective by base change), we deduce from \Cref{cor composition effective implies the first one effective} that $\cB G\to\spec(L')$ is effective, hence from \Cref{Teo BG effective implies G is a central ext of a torus} we have that $G$ is a central extension of a split torus by a finite group. Up to taking a refined cover, we can assume that the finite group is split, hence by \Cref{lemma extend line bundles quotient} the pullback homomorphism $\Pic([\spec(A)/G])\to\Pic(\cB G)$ is an isomorphism.

If we denote by $\cG' :=\cG\times_{\spec(L')}\spec(L)$ and by $\pi: \cG' \to \cG $ the first projection, then $ \cG'\cong \cB G $ is an isomorphism as the residual gerbe pulls back to the residual gerbe; then $\pi^*\cL_i$ extends to a line bundle on $[\spec(A)/G]$ from \Cref{lemma extend line bundles quotient}. We denote these extensions by $\widetilde{\cL_1}, ...,\widetilde{\cL_n}$. 

Consider now the line bundles $\cL_i':=\otimes_{\gamma \in \Gamma}\gamma^*\widetilde{\cL}_i$: these are by construction invariants, hence they descend to $\cX$ along the $\Gamma$-torsor $[\spec(A)/G]\to \cX$. Consider then the $\Gm^n$-torsor on $\cX$ induced by $\cL'_1,\ldots,\cL'_n$: we only need to check that this torsor is DM.

For this, is enough to check that it is DM when restricted to the residual gerbe: this is the case, because as $\widetilde{\cL_i}$ originally comes from $\cG$, it is $\Gamma$-invariant once restricted to $\cG'$, i.e. $\gamma^*\widetilde{\cL}_i|_{\cG'}=\widetilde{\cL}_i|_{\cG'}=\pi^*\cL_i$; this implies that $\cL_i'|_{\cG}\simeq \cL_i^{\otimes |\Gamma|}$, hence the associated $\Gm^n$-torsor is DM because the $\Gm^n$-torsor associated to $\cL_1,\ldots,\cL_n$ is so.
\end{proof}

\subsection{Inductive step: extending a generating set}\label{ss extending gset}In this subsection prove \Cref{teo intro}.

\begin{Lemma}\label{lemma extending gen set}
    Consider a smooth algebraic stack $\cX:=[\spec(A)/G]$ with good moduli space $\cX\to X=\spec(A^G)$, where $G$ is effective and split and $A^G$ is local. Let $U$ be an open subscheme of $\spec(A^G)$. Assume that $\cU:= [\spec(A)/G]\times_{\spec(A^G)}U$ admits $n$ generating line bundles $\{\widetilde{\cL}_i\}_{i=1}^n$. Then there are $n+m$ generating line bundles $\{\cL_i\}$ on $\cX$, such that $(\cL_i)|_{\cU}\cong \widetilde{\cL}_i$ if $i\le n$, and $(\cL_i)|_{\cU}\cong \cO_{\cU}$ otherwise.
\end{Lemma}
\begin{proof}
    As $\cX$ is smooth, we have a surjection $\Pic(\cX)_{\bQ}\to \Pic(\cU)_{\bQ}$. In particular, we can extend the generating line bundles $\widetilde{\cL}_1,\ldots,\widetilde{\cL}_n$ on $\cU$ to line bundles $\cL_1,\ldots,\cL_n$ on $\cX$.

    By \Cref{lemma extend line bundles quotient} we have an isomorphism $\Pic(\cX)\overset{\cong}{\to}\Pic(\cB G)$: in particular, as the $\bQ$-vector space $\Pic(\cB G)_{\bQ}$ is finitely generated, we must have $\Pic(\cX)_{\bQ}\simeq \bQ^{n+m}$ for some value of $m$. Complete $\cL_1,\ldots,\cL_n$ to a basis $\cL_1,\ldots,\cL_{n+m}$ of $\Pic(\cX)_{\bQ}$; we can pick the $\cL_i$ in such a way that $\cL_i|_{\cU}$ is trivial for $i>n$.
    
    Observe moreover that $\cB G$ has a generating set of line bundles because of \Cref{lemma la conj holds for gerbes}. Then, as the restriction of $\cL_1,\ldots,\cL_{n+m}$ generate the rational Picard group of $\cB G$, by \Cref{lm:generating generate} they also form a generating set for $\cB G$ in the sense of \Cref{def:generating line bundles}, i.e. the torsor associated to the restriction of $\cL_1\oplus\cdots\oplus\cL_{n+m}$ is DM. By upper-semicontinuity of the dimension of the stabilizer groups for $G$-actions, the torsor associated to $\cL_1\oplus\cdots\oplus\cL_{n+m}$ must be DM too.
    \end{proof}
\begin{Prop}\label{prop passo induttivo generating set}
    Let $\cX\to \spec(R)$ be a good moduli space with $R$ local ring and $\cX$ smooth. Let $p$ be the closed point of $\spec(R)$ and $U=\spec(R)\smallsetminus\{p\}$ be it its complement. Set $\cU:=U\times_{\spec(R)}\cX$, and suppose that $\cU$ admits a generating set. Then also $\cX$ admits a generating set.
\end{Prop}
Before proceeding with the proof, we recall the following fact, which is a slight improvement of Zariski's main theorem.
\begin{Teo}\label{teo zmt enhanced}
    Let $X$ be a separated DM stack and $Y\to X$ an \'etale, quasi-finite surjective morphism from a Noetherian scheme $Y$. Then there is a finite morphism $F\to X$ such that the second projection $Y\times_X F\to F$, when restricted to the connected components of $Y\times_X F$, is an open embedding.
\end{Teo}
The proof is essentially the proof of \cite[Theorem 4.5.1]{Alp21} (see in particular the second diagram of page 178). Note that among the hypotheses in \textit{loc. cit} it is required the stack $X$ being of finite type, but the whole proof works as well with the hypotheses above.
\begin{proof}[Proof of \Cref{prop passo induttivo generating set}] First observe that the stabilizer of the geometric closed point of $\cX$ is an effective group, which then is a central extension of a split torus by a finite split group (\Cref{Teo BG effective implies G is a central ext of a torus}). So there is a finite extension $L$ of the residue field of $R$, such that the residual gerbe of $\cX\times_{\spec(R)}\spec(L)$ is $\cB G$, where $G$ is an effective group. From the main result of \cite{luna}, there is a cartesian diagram as follows, with \'etale surjective horizontal arrows, where $A^G$ has a unique closed point:
$$\xymatrix{[\spec(A)/G]\ar[d]\ar[r] & \cX\ar[d] \\ \spec(A^G) \ar[r]^j & \spec(R).}$$
We can now apply \Cref{teo zmt enhanced} to $j$ to get the following diagram, where every square is cartesian
$$\xymatrix{\cF' \ar[rr] \ar[dd]\ar[dr]^{i_F} & & [\spec(A)/G]\ar[dr]^i\ar[dd] &\\
 & \cF\ar[dd]\ar[rr] & &\cX\ar[dd]_\pi \\
 \spec(R_F\otimes_R A^G)\ar[dr]^-{j_F}\ar[rr] & & \spec(A^G)\ar[dr]^-j & \\
 & \spec(R_F) \ar[rr] & & \spec(R).}$$
 We denote by $\pi_F$ the morhphism $\cF\to \spec(R_F)$; and up to replacing $R_F$ with the normalization of one of the connected components of $\spec(R_F)$ which dominate $\spec(R)$, we can assume that $R_F$ is normal. Moreover, up to further extending, we can assume that $\spec(R_F)\to\spec(R)$ generically is Galois with Galois group $\Gamma$. In particular, $\Gamma$ acts on $\spec(R_F)$, $\spec(R_F\otimes_R A^G)$, $\cF$ and $\cF'$ in a way such that all the arrows in the diagram above are equivariant. Finally, by assumption, $i_F$ and $j_F$ are open embeddings once restricted to each connected component of their domains. We first plan on descending a generating set on $\cF'$ to $\cF$, using \Cref{lemma extending gen set}.

 First, we restrict $\pi$ to $\cU$. Let $m$ be the number of line bundles in the generating set of $\cU$. We can pull back this set to $\cU_A:= \cU\times_{\cX}[\spec(A)/G]$, which is by construction a quotient stack by $G$, and extend it to a generating set for $[\spec(A)/G]$ using \Cref{lemma extending gen set}; now we have $n+m$ line bundles on $[\spec(A)/G]$ which are a generating set. We can pull back this generating set to $\cF'$, and now  so we have $n+m$ line bundles $\cL_1,...,\cL_{n+m}$ on $\cF'$. 

  We now study the action of $\Gamma$ on $\cL_i$. Consider the inclusion of the closed residual gerbe
  $\cB G_L\to [\spec(A)/G]$, where $L$ is a field, and its pull-back $\bigsqcup_{\ell = 1}^f \cB G'_{L'} = \cB G_L\times_{[\spec(A)/G]}\cF'$.
 There is an action of $\Gamma$ on $\bigsqcup_{\ell = 1}^f \cB G'_{L'}$, which acts transitively on its topological space. From \Cref{lemma extend line bundles quotient}, there are injective morphisms $\Pic(\cF')\to \bigoplus_{\ell = 1}^f  \Pic(\cB G'_{L'})$ so one can understand
 the action of $\Gamma$ on $\cL_i$ by understanding its action on $\bigsqcup_{\ell = 1}^f \cB G'_{L'}$. In particular, $\Gamma$ acts trivially on $\cL_i$, as they are pulled back from $[\spec(A)/G]$, namely $\bigotimes_{\gamma \in \Gamma}\gamma^*\cL_{i} = \cL_i^{\otimes|\Gamma|}$. Therefore $\{\cL_i\}$ is a $\Gamma$-invariant generating set and the characters induced by $\cL_i$ on each connected component of $\bigsqcup\cB G'_{L'}$  are the same.

 The scheme $\spec(R_F)$ has an open subset $U_F=U\times_{\spec(R)}F$ and $d$ closed points $p_1,...,p_d$, and the action of $\Gamma$ is transitive on the closed points. For each of the points $p_i$, we can take an open subscheme $\spec(R_i)$ of a connected component of $\spec(R_F\otimes_R A^G)$ such that the composition $\spec(R_i)\to \spec(R_F)$ is the localization of $\spec(R_F)$ at $p_i$. So each $\cF_i:=\spec(R_i)\times_{\spec(R_F\otimes_R A^G)}\cF'$ has $n+m$ generating line bundles,
 and we plan on gluing those along the generic points of $\spec(R_i)$ for every $i$ to get $n+m$ generating line bundles on $\cF$. But the $n+m$ generating line bundles we choose glue along the generic fiber of $\cF_i\to \spec(R_i)$: the first $n$ of them came from line bundles on $\cX\times_{\spec(R)}\cU$, and the other $m$ restrict to the trivial line bundle generically. So $\cF$ admits a generating set, we denote it by $\cG_1,...,\cG_{n+m}$, where the first $m$ line bundles are the pull-back of a generating set of $\cX\times_{\spec(R)}\cU$.

 We now plan on descending a generating set from $\cF$ to $\cX$. Recall that we have an action of $\Gamma$ on $\spec(R_F)$, so we have the following diagram, where all squares are cartesian:
 $$\xymatrix{\cF \ar[d] \ar[r] &\cX' \ar[r] \ar[d] & \cX\ar[d]^\pi \\ \spec(R_F)\ar[r] & [\spec(R_F)/\Gamma]\ar[r] & \spec(R).}$$
 The data of a line bundle on $\cX'$ is the data of a $\Gamma$-equivariant line bundle on $\cF$. So for $i\le n$ we define $\cG_i':=\cG_i$ and for $1\ge i \ge n$ we define $\cG_{m+i}':=\bigotimes_{\gamma \in \Gamma}\gamma^*\cG_{m+i}$. Now the line bundles $\cG_i'$ admit an action by $\Gamma$. To check that $\{\cG_{m+i}'\}$ is still a generating set we pull them back to $\cF$. Since $\cL_i$ was a $\Gamma$-\textit{invariant} generating set, the pull-backs of $\cG_i'$ forms a generating set.
 
 So they descend to $n+m$ line bundles on $\cX'$, which we denote by $\cH_1,...,\cH_{n+m}$. Since $\cX$ is smooth, the scheme $\spec(R)$ is normal, so $[\spec(R_F)/\Gamma]\to \spec(R)$ is a coarse moduli space. So up to replacing $\cH_i$ with some powers, they descend to a generating set on $\cX$ as desired.\end{proof}

\begin{theorem}\label{teo effettivo implica ho il gen set}
Assume that $\cX$ is a smooth algebraic stack admitting a good moduli space $\pi:\cX\to X$. If $\pi$ is effective, then $\cX$ has a generating set, namely it admits a morphism $\cX\to \cB \Gm^n$ for $n$ big enough, which is representable in DM stacks.
\end{theorem}
\begin{proof}Consider the following set
$$\sU:=\{\cU \subseteq \cX:\cU\text{ is open and admists a generating set}\}.$$
It is ordered by inclusion, and every two elements $\cU_1,\cU_2\in \sU$ admit $\cU_3$ which contains both. Indeed, since $\cX$ is smooth, we can extend each line bundle in the generating sets of $\cU_1$ and $\cU_2$ to $\cU_1\cup \cU_2$. Consider now the direct sum of direct sum of the extensions of the line bundles coming from $\cU_1$ and $\cU_2$: the associated $\Gm^{n_1+n_2}$-torsor will be DM both when restricted to $U_1$ and when restricted to $U_2$, thus $\cU_1\cup \cU_2\in \sU$. 

 Observe also that $\sU$ is not empty, as from \Cref{prop la conj holds for 0 dim gms} we have a generating set over the generic fiber of $\pi$, which we can spread out.

Consider then $\cU$ a maximal element of $\sU$, and assume by contradiction that $\cU\neq\cX$. Let $p\in \cX\smallsetminus \cU$ be a generic point of an irreducible component of $ \cX\smallsetminus \cU$. From \Cref{prop passo induttivo generating set} the stack $\cX\times_{X}\spec(R)$ admits a generating set, which we can spread out to a generating set of $\cU'$, an open subset of $\cX$ \textit{containing} $p$. Then $\cU\cup \cU'$ admits a generating set, which contradicts the maximality of $\cU$. So $\cU=\cX$ as desired. 
\end{proof}

\begin{Remark}
    Observe that if $\cX$ admits a DM morphism $\cX\to \cB \Gm^n$, then its moduli space $\cX\to X$ is always effective if $X$ is separated; while for the vice versa we need some assumptions on the singularities of $\cX$, to extend the line bundles. For example, there are examples of $\Gm$-gerbes $\cX\to X$ over singular $X$ whose class in $\oH^2(X,\Gm)$ is not torsion (see for example \cite[II.1.11.b]{gr_brauer_thing}); those would be effective, but not a global quotient of a DM stack by a torus. 
    To check that any $\Gm$-gerbe is effective, it suffices to check the following. Given any $\Gm$-gerbe $\cX'\to \cT$, where $\cT$ is a DM gerbe, over the spectrun of a DVR $R$ which is trivial over a dense open substack $\cU\subseteq \cT$, is indeed trivial. But then $\cX'\to \cT$ is trivial: the map $\oH^2(\cT,\Gm)\to \oH^2(\cU,\Gm)$ is injective. So $\cX'\cong \cB \Gm\times \cZ$ and the second projection is effective. On the other hand $\cX$ cannot be the quotient of a DM stack by a torus. Otherwise, there would be a map $\cX\to \cB \Gm^n$. This corresponds to $n$ line bundles on $\cX$ which are twisted, and up to replacing some these line bundles with an appropriate multiple, they would be $m$-twisted for $m$ big enough. But then its determinant would be an $nm$ twisted line bundle, so $\cX\to X$ would be $nm$-torsion from \cite{lieblich2008twisted}*{Propisition 3.1.1.8}. It would be interesting to give a finer characterization of the singularities on $\cX$ for which the thesis of \Cref{teo effettivo implica ho il gen set} still holds. 
\end{Remark}

\bibliographystyle{amsalpha}
\bibliography{bibliography}

\begin{bibdiv}
\begin{biblist}

\bib{AHH}{misc}{
      author={Alper, Jarod},
      author={Halpern-Leistner, Daniel},
      author={Heinloth, Jochen},
       title={Existence of moduli spaces for algebraic stacks},
        date={2018},
        note={available at \url{https://arxiv.org/abs/1812.01128}},
}

\bib{luna}{article}{
      author={Alper, Jarod},
      author={Hall, Jack},
      author={Rydh, David},
       title={A luna {\'e}tale slice theorem for algebraic stacks},
        date={2020},
     journal={Annals of Mathematics},
      volume={191},
      number={3},
       pages={675\ndash 738},
}

\bib{alp}{article}{
      author={Alper, Jarod},
       title={Good moduli spaces for {A}rtin stacks},
        date={2013},
     journal={Ann. Inst. Fourier (Grenoble)},
      volume={63},
      number={6},
       pages={2349\ndash 2402},
}

\bib{Alp21}{article}{
      author={Alper, Jarod},
       title={Introduction to stacks and moduli},
        date={2023},
     journal={Available on the author's webpage},
}

\bib{AOV}{article}{
      author={Abramovich, Dan},
      author={Olsson, Martin},
      author={Vistoli, Angelo},
       title={Twisted stable maps to tame {A}rtin stacks},
        date={2011},
     journal={J. Algebraic Geom.},
      volume={20},
      number={3},
       pages={399\ndash 477},
}

\bib{brionext}{article}{
      author={Brion, Michel},
       title={On extensions of algebraic groups with finite quotient},
        date={2015},
     journal={Pacific J. Math},
      volume={279},
      number={1-2},
       pages={135\ndash 153},
}

\bib{bresciani2022arithmetic}{article}{
      author={Bresciani, Giulio},
      author={Vistoli, Angelo},
       title={An arithmetic valuative criterion for proper maps of tame
  algebraic stacks},
        date={2022},
     journal={arXiv preprint arXiv:2210.03406},
}

\bib{conrad2007arithmetic}{article}{
      author={Conrad, Brian},
       title={Arithmetic moduli of generalized elliptic curves},
        date={2007},
     journal={Journal of the Institute of Mathematics of Jussieu},
      volume={6},
      number={2},
       pages={209\ndash 278},
}

\bib{Con14}{article}{
      author={Conrad, Brian},
       title={Reductive group schemes},
        date={2014},
     journal={Autour des sch{\'e}mas en groupes},
      volume={1},
      number={93-444},
       pages={24},
}

\bib{EG}{article}{
      author={Edidin, Dan},
      author={Graham, William},
       title={Equivariant intersection theory},
        date={1998},
     journal={Invent. Math.},
      volume={131},
      number={3},
       pages={595\ndash 634},
}

\bib{edidin2001brauer}{article}{
      author={Edidin, Dan},
      author={Hassett, Brendan},
      author={Kresch, Andrew},
      author={Vistoli, Angelo},
       title={Brauer groups and quotient stacks},
        date={2001},
     journal={American Journal of Mathematics},
      volume={123},
      number={4},
       pages={761\ndash 777},
}

\bib{Arithmetictoricvars}{article}{
      author={Elizondo, E~Javier},
      author={Lima-Filho, Paulo},
      author={Sottile, Frank},
      author={Teitler, Zach},
       title={Arithmetic toric varieties},
        date={2014},
     journal={Mathematische Nachrichten},
      volume={287},
      number={2-3},
       pages={216\ndash 241},
}

\bib{grothendieck_Serre_panin}{article}{
      author={Fedorov, Roman},
      author={Panin, Ivan},
       title={A proof of the {G}rothendieck-{S}erre conjecture on principal
  bundles over regular local rings containing infinite fields},
        date={2015},
        ISSN={0073-8301,1618-1913},
     journal={Publ. Math. Inst. Hautes \'{E}tudes Sci.},
      volume={122},
       pages={169\ndash 193},
         url={https://doi.org/10.1007/s10240-015-0075-z},
      review={\MR{3415067}},
}

\bib{giraud2020cohomologie}{article}{
      author={Giraud, Jean},
       title={Cohomologie non ab\'{e}lienne; pr\'{e}liminaires},
        date={1965},
        ISSN={0001-4036},
     journal={C. R. Acad. Sci. Paris},
      volume={260},
       pages={2392\ndash 2394},
      review={\MR{201489}},
}

\bib{grothendieck1968groupe}{article}{
      author={Grothendieck, Alexander},
       title={Le groupe de brauer. ii. th{\'e}orie cohomologique},
        date={1968},
     journal={Dix expos{\'e}s sur la cohomologie des sch{\'e}mas},
      volume={3},
       pages={67\ndash 87},
}

\bib{gr_brauer_thing}{article}{
      author={Grothendieck, Alexander},
       title={Le groupe de brauer. ii. th{\'e}orie cohomologique},
        date={1968},
     journal={Dix expos{\'e}s sur la cohomologie des sch{\'e}mas},
      volume={3},
       pages={67\ndash 87},
}

\bib{kollar1997quotient}{article}{
      author={Koll{\'a}r, J{\'a}nos},
       title={Quotient spaces modulo algebraic groups},
        date={1997},
     journal={Annals of mathematics},
      volume={145},
      number={1},
       pages={33\ndash 79},
}

\bib{lieblich2008twisted}{article}{
      author={Lieblich, Max},
       title={Twisted sheaves and the period-index problem},
        date={2008},
     journal={Compositio Mathematica},
      volume={144},
      number={1},
       pages={1\ndash 31},
}

\bib{LMB}{book}{
      author={Laumon, G\'{e}rard},
      author={Moret-Bailly, Laurent},
       title={Champs alg\'{e}briques},
      series={Ergebnisse der Mathematik und ihrer Grenzgebiete. 3. Folge. A
  Series of Modern Surveys in Mathematics [Results in Mathematics and Related
  Areas. 3rd Series. A Series of Modern Surveys in Mathematics]},
   publisher={Springer-Verlag, Berlin},
        date={2000},
      volume={39},
        ISBN={3-540-65761-4},
      review={\MR{1771927}},
}

\bib{Mil}{book}{
      author={Milne, James~S},
       title={Algebraic groups: the theory of group schemes of finite type over
  a field},
   publisher={Cambridge University Press},
        date={2017},
      volume={170},
}

\bib{rydh2011etale}{article}{
      author={Rydh, David},
       title={{\'E}tale d{\'e}vissage, descent and pushouts of stacks},
        date={2011},
     journal={Journal of Algebra},
      volume={331},
      number={1},
       pages={194\ndash 223},
}

\bib{Galois_cohom}{book}{
      author={Serre, Jean-Pierre},
       title={Galois cohomology},
   publisher={Springer-Verlag, Berlin},
        date={1997},
        ISBN={3-540-61990-9},
  url={https://doi-org.offcampus.lib.washington.edu/10.1007/978-3-642-59141-9},
        note={Translated from the French by Patrick Ion and revised by the
  author},
      review={\MR{1466966}},
}

\bib{stacks-project}{misc}{
      author={{Stacks project authors}, The},
       title={The stacks project},
         how={\url{https://stacks.math.columbia.edu}},
        date={2022},
}

\end{biblist}
\end{bibdiv}

\end{document}